\documentclass[notitlepage,11pt, a4paper]{article}
\usepackage{booktabs}
\usepackage{subfigure, epsfig, epstopdf}
\usepackage{graphicx,amssymb}
\usepackage{multirow, multicol}
\usepackage{amsmath}
\usepackage{amsthm}
\usepackage{color}
\usepackage{setspace}
\usepackage[left=1in,top=1in,right=1in,bottom=1in, nohead]{geometry}
\usepackage[colorlinks=true,backref=true,hyperindex,citecolor=blue,linkcolor=blue]{hyperref}
\usepackage{natbib}
\usepackage[utf8]{inputenc}
\usepackage{gensymb}
\usepackage{pdflscape}
\usepackage{courier}
\usepackage{epsfig,epstopdf}
\usepackage{booktabs}
\usepackage{natbib}
\usepackage{chngpage}
\usepackage{algorithm}
\usepackage{mathtools}
\usepackage{longtable}
\usepackage{algorithmic,color}
\usepackage[titletoc]{appendix}
\usepackage{float}
\usepackage{textcomp}

\bibpunct{(}{)}{;}{a}{}{,}
\newtheorem{theorem}{Theorem}
\newtheorem{lemma}{Lemma}

\newtheorem{assumption}{Assumption}
\newtheorem{lemmastar}[lemma]{{*}Lemma}
\newtheorem*{acknowledgement}{Acknowledgement}

\newenvironment{definition}[1][Definition]{\begin{trivlist}
  \item[\hskip \labelsep {\bfseries #1}]}{\end{trivlist}}

 \graphicspath{{figure/}} 

\title{Index-Based Policy for Risk-Averse Multi-Armed Bandit}

\date{\today}

\begin{document}
\begin{titlepage}\thispagestyle{empty}
\maketitle
\noindent\textbf{Authors}
\vspace{3 ex}

\noindent Jianyu Xu: Department of Industrial Systems Engineering and Management, National University
of Singapore, 1 Engineering Drive 2, Singapore.

\noindent William B. Haskell: Department of Industrial Systems Engineering and Management, National University
of Singapore, 1 Engineering Drive 2, Singapore.

\noindent Zhisheng Ye: Department of Industrial Systems Engineering and Management, National University
of Singapore, 1 Engineering Drive 2, Singapore.
\vspace{6 ex}

\noindent\textbf{Corresponding author}
\vspace{3 ex}

\noindent William B. Haskell: wbhaskell@gmail.com.

\end{titlepage}
\begin{abstract}
The multi-armed bandit (MAB) is a classical online optimization model
for the trade-off between exploration and exploitation. The traditional
MAB is concerned with finding the arm that minimizes the mean cost.
However, minimizing the mean does not take the risk of the problem into
account. We now want to accommodate risk-averse decision makers.
In this work, we introduce a coherent risk measure as the criterion to form a risk-averse MAB.
In particular, we derive an index-based online sampling framework for the risk-averse MAB.
We develop this framework in detail for three specific risk measures, i.e.
the conditional value-at-risk, the mean-deviation and the shortfall
risk measures. Under each risk measure, the convergence rate for the upper bound on the
pseudo regret, defined as the difference between the expectation of
the empirical risk based on the observation sequence and the
true risk of the optimal arm, is established.
\end{abstract}


\section{Introduction}

The multi-armed bandit (MAB) is a classical problem named after an imaginary slot machine with multiple arms \citep{robbins1985some}.
At each time step, a player pulls one of the arms and receives a random cost associated with the chosen arm.
A policy, or a strategy for the player, denoted by $\varphi$, is an algorithm to choose the next arm based on the past choices and the observed cost sequence. Given a sequence of pulls from the arms,
the pseudo regret, which measures the performance of the policy, is defined as the difference between the expectation of the average cost after $n$ time steps and the lowest mean cost among the arms.
The player's objective is to design a policy that minimizes the pseudo regret either asymptotically \citep{lai1985asymptotically} or uniformly over time \citep{auer2002finite}.
The above basic MAB has many variations depending on the application.
Recent work on this topic has led to many important theoretical results \citep{agrawal1988asymptotically,bubeck2012regret} as well as interesting applications
\citep{jennison1999group,mohri2014optimal,scott2015multi}. The asymptotically optimal policy is usually obtained by the index-based policy proposed in \cite{lai1985asymptotically},
and extended by \cite{agrawal1995sample}, \cite{auer2002finite} and \cite{kleinberg2005nearly}. An index-based policy calculates an index for each arm at each time step. An index consists of two parts, the empirical estimation of the mean and another term related to the probability confidence bound from the estimation. The arm with the lowest current index is chosen each time.

The classical MAB defines the pseudo regret in terms of the long-term average cost, and thus it is risk-neutral.
The risk-neutral MAB does not take the variance of the random cost of each arm into consideration.
As a result, it is inappropriate in many applications which require reliability guarantees.
In clinical trials, for example, a specific treatment with a low average side effect can be highly variable among different
patients and thus it may cause serious medical problems for an individual.
These applications require a criterion that takes the variation/spread/dispersion of the arms into account.
It is thus natural to select a risk measure and then to try to choose the arm with the lowest risk.
Some popular risk measures include value-at-risk \citep{benati2007mixed}, conditional value-at-risk (CVaR) \citep{rockafellar2000optimization} and expected shortfall \citep{acerbi2002expected}. Risk measures have been extensively studied in decision analysis \citep{follmer2002convex,brown2009satisficing,brown2012aspirational}, reinforcement
learning \citep{mannor2011mean,shen2013risk} and operations research \citep{ruszczynski2006optimization,liu2017distributionally}.

The risk-averse MAB has attracted significant attention in the recent literature.
\cite{sani2012risk} uses the mean-variance risk measure to assess the performance of each arm.
They proposed an index-based policy and proved a sub-linear upper
bound for the pseudo regret. \cite{maillard2013robust} extends the mean-variance measure in \cite{sani2012risk} to a more general risk measure that takes the tail of the cumulative distribution function (CDF) into consideration. \cite{galichet2013exploration} introduces the conditional
value-at-risk to MAB. However, their theoretical analysis only
focuses on the special case where the CVaR degenerates to the essential
infimum. These studies adopt the index-based policy inherited from
the classical risk-neutral MAB. The risk measures involved in existing work are limited and many popular risk measures are not included. So far, there is a lack of research on risk-averse MAB and
there is not yet a consensus on the notion of pseudo regret. In the risk-neutral MAB, the loss of a single choice is measured as the difference between the means of the chosen arm and the optimal arm. The total regret is the summation of the losses of each single choice in a sequence. The reason is because the total cost is the summation of the cost drawn each time. So, the total regret is also additive in terms of each single loss. However, when we use a risk measure as the criterion in the risk-averse MAB, the property of additivity no longer holds. This leads to the requirement of a rational pseudo regret for the risk-averse MAB.

In this work, we formalize the risk-averse MAB by generalizing the classical risk-neutral MAB.
We introduce a general coherent risk measure in the MAB as the criterion for the optimal arm.
We define the pseudo regret as the expectation of the difference between the empirical risk measure based on the observation sequence and the true risk measure of the optimal arm.
Further, we conclude that by using our notion of pseudo regret as the target function for optimization, we are still correctly searching for the single optimal arm. An index-based policy is then constructed to find the optimal risk-averse arm.
To prove the convergence of the pseudo regret under the proposed policy, we restrict the general risk measure to three specific cases: that is, the CVaR, the mean-deviation (MD) and the shortfall risk measure.
These risk measures have not yet been fully investigated for the risk-averse MAB.
To show the performance of the policy in the long run, we present a detailed theoretical analysis of the
convergence rate of the pseudo regret for each of these three risk measures.

This paper is organized as follows. Section \ref{Section Risk-Averse MAB} formalizes the risk-averse MAB by using a general coherent risk measure as the decision maker's criterion for arm selection. We define the pseudo regret and then construct an index-based policy for a general risk-averse MAB. In Section \ref{Main results}, we introduce our three specific risk measures of interest. Then, we present the main results of this work; that is,
the expression of the index and the corresponding convergence rate of the pseudo regret for each of these three risk measures. Section \ref{Proof of the main results} provides the detailed proofs of our main results. Some concluding remarks are given in Section \ref{Concluding remarks}.

\section{Risk-Averse MAB}\label{Section Risk-Averse MAB}

Consider an MAB with $K$ arms in total.
Each pull of an arm $k$, $k=1,\ldots,K$,
generates a realization of a nonnegative random cost $X_{k}$ with
mean $\mu_{k}\triangleq\mathbf{E}(X_{k})$ and cumulative distribution function $F_{k}$. In the remainder of this work, we suppose all $X_{k}$'s
are essentially bounded by $M$, i.e. $\mathbf{P}\{ X_{k}\in[0,\:M]\} =1$
for all $k=1,\ldots,K$. This assumption is commonly used in the MAB literature,
see \cite{sani2012risk} and \cite{maillard2013robust} among others.
We further suppose that successive pulls from an arm yield a sequence of i.i.d.
random costs and each pull does not change the distributions of the $K$ arms.

We operate on a finite time horizon of $n$ time steps.
A policy $\varphi$ generates a sequence of choices $\{I_{t}\}_{t=1}^{n}$ where at each time $t\in\{1,\ldots,n\}$, arm $I_{t}\in\{ 1,\ldots,K\}$ is chosen. Let the number of observations from arm $k$ up to time $n$ be $T_{k}(n)$, i.e. $T_{k}(n)\triangleq\sum_{t=1}^{n}I\{I_{t}=k\}$,
where $I\{\cdot\}$ is the indicator function. The risk-neutral
MAB problem seeks the arm $k^{*}$ with the lowest mean
cost, i.e. $k^{*}\triangleq\mathrm{argmin}_{k=1,\ldots, K}\mu_{k}$. Equivalently,
it attempts to minimize the following pseudo regret \citep{lai1985asymptotically}
\begin{equation}\label{eq:neutral mab}
\max_{k=1,\ldots,K}\left\{ \mathbf{E}(\frac{1}{n}\sum_{t=1}^{n}x_{I_{t},\,t})-\mu_{k}\right\} =\frac{1}{n}\sum_{k=1}^{K}\mathbf{E}T_{k}\left(n\right)\left(\mu_{k}-\mu_{k^{*}}\right).
\end{equation}

\subsection{Risk-averse MAB formulation}\label{Risk-averse MAB formulation}

To formalize the risk-averse MAB, we introduce a coherent risk-measure $\rho$ as the decision maker's objective instead of the expectation.
A coherent risk measure is defined as follows.
\begin{definition}\label{def" risk measure}
\citep{ruszczynski2006optimization} Let $\mathcal{L}$ be a space
of essentially bounded random variables. A risk measure $\rho:\:\mathcal{L}\rightarrow\mathbb{R}$
is called a coherent risk measure if it satisfies:
\begin{itemize}
\item[(1.1)] Convexity: For all $X,\:Y\in\mathcal{L}$ and $\forall\lambda\in\left[0,\:1\right]$,
$\rho\left[\lambda X+\left(1-\lambda\right)Y\right]\leq\lambda\rho\left(X\right)+\left(1-\lambda\right)\rho\left(Y\right)$;

\item[(1.2)] Monotonicity: If $X,\:Y\in\mathcal{L}$ and $X\leq Y$, then
$\rho\left(X\right)\leq\rho\left(Y\right)$;

\item[(1.3)] Translation equivalence: If $\alpha\in\mathcal{\mathbb{R}}$
and $X\in\mathcal{L}$, then $\rho\left(X+\alpha\right)=\rho\left(X\right)+\alpha$;

\item[(1.4)] Positive homogeneity: If $\beta>0$ and $X\in\mathcal{L}$,
then $\rho\left(\beta X\right)=\beta\rho\left(X\right)$.

\end{itemize}
\end{definition}

Let $\rho_{k} \triangleq \rho (X_{k})$ denote the risk of arm $k$. In the risk-averse MAB, we define the optimal
arm to be the one with the lowest risk, i.e. $k^{*}\triangleq\mathrm{argmin}_{k=1,\ldots, K}\rho_{k}$.
We assume that the optimal arm $k^{*}$ is always unique. We call arm $k$ a sub-optimal arm if $k \neq k^{*}$.
This assumption is reasonable in practice, as it is uncommon for risk measures for two different populations to coincide. Three specific risk measures will be introduced and studied in Section \ref{Main results}.

Our first objective is to extend the definition of the pseudo regret
from the risk-neutral MAB to the risk-averse MAB. From \eqref{eq:neutral mab}, we note that the risk-neutral MAB aims to minimize the expectation of the difference between the empirical mean of the whole observation sequence and the mean of the optimal arm.
Similarly, in the risk-averse case, we can let $\widehat{\rho}_{\varphi,\,n}$ be the empirical risk based on $\{x_{I_{t},\,t}\}_{t=1}^{n}$.
Then, we can use the expectation of the difference between $\widehat{\rho}_{\varphi,\,n}$ and $\rho_{k^{^{*}}}$  as the pseudo regret. A formal definition of the pseudo regret is given below.
\begin{definition}
The pseudo regret of any given policy $\varphi$ under a risk measure $\rho$
is defined as
\begin{equation}\label{eq:pseudoregret}
R_{n}(\varphi)\triangleq\mathbf{E}\widehat{\rho}_{\varphi,\,n}-\rho_{k^{^{*}}}.
\end{equation}
\end{definition}
Minimizing $R_{n}(\varphi)$ yields the arm with the lowest risk asymptotically.
Throughout this work, we call $\widehat{\rho}_{\varphi,\,n}$ the empirical risk measure of the policy $\varphi$.

\subsection{Algorithm for the risk-averse MAB}\label{Index policy for risk-averse MAB}

In this section, we present our main algorithm for the risk-averse MAB.
The main idea of the algorithm is based on the notion of the lower probability confidence bound.
Therefore, we call it the Risk-Averse Lower Confidence Bound (RA-LCB) algorithm.
To initialize, the algorithm chooses each arm once. This ensures that each term $T_{k}(n)$ is larger than 0 after initialization. The calculation of the index for each arm after initialization relies on the condition $T_{k}(n)>0$. At each time $n>K$, the algorithm calculates an index for each arm $k$, $k=1,\ldots,K$, which is the difference between two terms $\widehat{\rho}_{k,\,n}$ and $\varepsilon_{\rho}(\cdot)$. The first term is the empirical estimate of $\rho_{k}$. The second term $\varepsilon_{\rho}(\cdot)$ is a function of $n$, $T_{k}(n)$, $K$ and the confidence level $\delta\in (0,\,1)$. It is related to the probability confidence bound of the empirical estimate $\widehat{\rho}_{k,\,n}$. This second term ensures that the true risk of an arm falls above its index with an overwhelming probability as $n$ becomes larger. This is the basis of the convergence proof for the pseudo regret. Let $S_{k}(n)\triangleq\{t:I_t=k,t\leq n\}$ be the set of all times at which arm $k$ is chosen. Then $X_{k,\,t}$, $t\in S_{k}(n)$ is the observation sequence from arm $k$. Both terms in the index are functions of the sequence $X_{k,\,t}$, $t\in S_{k}(n)$. The specific forms of these two terms vary under different risk measures. We will specify both terms under the three risk measures that we consider in the sections below. As in the risk-neutral case, at each time step, the arm with the lowest index is chosen. A summary of the main flow of our procedure is given in Algorithm 1. Throughout the rest of the paper, we fix the confidence level $\delta\in(0,\,1)$.

\begin{algorithm}\label{algorithm 1}
\caption{Risk-Averse Lower Confidence Bounds (RA-LCB)}
\textbf{Input:} constant $\delta$, $K$, a confidence bound function $\varepsilon_{\rho}(\cdot)$.\\
\textbf{Initialization}: Choose each arm once in the first $K$ pulls.\\
\textbf{while} $n>K$

\quad \textbf{for} $k=1:K$ \textbf{do}
\begin{enumerate}
\item Calculate $\widehat{\rho}_{k,\,n}$ based on $X_{k,\,t}$, $t\in S_{k}(n)$.
\item Calculate the index for each arm:
$B_{k,\,n} \triangleq \widehat{\rho}_{k,\,n}-\varepsilon_{\rho}(n,\,T_k(n),\,K,\delta)$.
\end{enumerate}
\quad \textbf{end for}\\
$I_{n}=\mathop{\arg\min}_{k}B_{k,\,n}$.\\
$n \Leftarrow n+1$.\\
pull arm $I_n$ at time $n$, $n=1,\,2,\cdots$.\\
\textbf{end while}
\end{algorithm}

\section{Main results}\label{Main results}

This section presents our main results. We apply Algorithm 1 to three specific risk measures: CVaR, MD and shortfall. For each risk measure, we derive the corresponding upper bound for its pseudo regret. We need the following condition when we establish the convergence rate of the pseudo regret for the three risk measures. The explicit form of $\varepsilon_{\rho}(\cdot)$ under each risk measure is given below. In Section \ref{Proof of the main results}, we will show that each specific $\varepsilon_{\rho}(\cdot)$ we give in this section satisfies this condition though the procedure of proof.

\noindent \textbf{Condition 1} For all $k=1,\ldots,K$, the function $\varepsilon_{\rho}(\cdot)$ satisfies
\begin{equation*}
\mathbf{P}\left\{ \left|\widehat{\rho}_{k,\,n}-\rho_{k}\right|\geq\varepsilon_{\rho}(n,\,T_{k}(n),\,K,\delta)\right\} \leq\frac{C\delta}{n^{2}K},
\end{equation*}
where $C$ is a constant.

The construction of $\varepsilon_{\rho}(\cdot)$ under each specific risk measure is based on the various forms of concentration results for $|\widehat{\rho}_{k,\,n}-\rho_{k}|$. Specifically, $\varepsilon_{\rho}(\cdot)$ is directly related to the confidence bounds derived from certain concentration inequalities. If the empirical risk $\widehat{\rho}_{k,\,n}$ is a summation of different terms, we derive the confidence bound on each term and use the summation of these bounds to achieve $\varepsilon_{\rho}(\cdot)$. Under Condition 1, the convergence rate of the pseudo regret for each risk measure is of the order $O(\sqrt{\log n / n})$. The upper bound for the pseudo regret under each risk measure is correspondingly given in Section 3.1-3.3. In Section 3.4, we make some discussion on the order $O(\sqrt{\log n / n})$ of the convergence rate.

\subsection{Conditional value-at-risk}\label{CVaR risk measure}

We first define CVaR as follows.
\begin{definition}
\textbf{(CVaR risk measure)} Consider a random variable $Y$ with CDF $F_Y$. For a fixed level $\alpha\in[0,\:1]$, let $F_{Y}^{-1}(\alpha)\triangleq\inf\{x:\:F_{X}(x)\geq\alpha\}$ be the $\alpha$-quantile of $Y$. Then, the CVaR at level $\alpha$ of arm $k$ is $\mathbf{E}[X_{k}\,|\,X_{k}\geq F_{k}^{-1}(\alpha)]$, i.e.
\begin{equation*}
\rho_{k}^{C} \triangleq (1-\alpha)^{-1}\int_{\alpha}^{1}F_{k}^{-1}(\tau)d\tau.
\end{equation*}
\end{definition}
When $\alpha=0$, CVaR becomes expectation. We need the following assumption for CVaR which ensures the convergence of the empirical estimate of the risk towards its actual value. Lemma \ref{lemma 1} below gives a convenient representation of CVaR and also relies on this assumption.
\begin{assumption}\label{CVaR assumption 1}
For all $k=1,\ldots,K$, $F_{k}$ is continuously differentiable on $(0,\:M)$ with corresponding density function $f_{k}$
and $[f_{k}(F_{k}^{-1}(\alpha)]^{-1}$, $k=1,\ldots,K$ is uniformly upper bounded by a constant $m(\alpha)$, $\forall\alpha\in(0,\;1)$.
\end{assumption}
This common assumption also appears in some statistical literature \citep{bahadur1966note,arcones1996bahadur}.
In this work, we use a convenient representation of CVaR \citep{rockafellar2000optimization} given as
\begin{equation}\label{eq:CVaRi}
\rho_{k}^{C}=\inf_{\eta\in\mathbb{R}}\{\eta+(1-\alpha)^{-1}\mathbf{E}[(X_{k}-\eta)_{+}]\}.
\end{equation}
As $\{\eta+(1-\alpha)^{-1}\mathbf{E}[(X_{k}-\eta)_{+}]\}$
is convex in $\eta$, the range of $\eta$ can be restricted to the
support of $F_{k}$, as shown in the following lemma.
\begin{lemma}\label{lemma 1}
\textup{[\citep{rockafellar2000optimization} Theorem 1]} Consider a random
variable $X$ with a continuous CDF $F_{X}$ and bounded support $[0,\:M]$. Let
\begin{equation*}
g\left(\eta\right)\triangleq\eta+(1-\alpha)^{-1}\mathbf{E}[(X-\eta)_{+}],
\end{equation*}
then $g\left(\eta\right)$ reaches its global minimal at $\eta=F_{X}^{-1}(\alpha)$. Further,
the CVaR of $X$ is
\begin{equation*}
F_{X}^{-1}(\alpha)+(1-\alpha)^{-1}\mathbf{E}\left[(X-F_{X}^{-1}(\alpha))_{+}\right].
\end{equation*}
\end{lemma}

From the sequence $X_{k,\,t}$, $t\in S_{k}(n)$ of arm $k$ as defined in Algorithm 1,
the empirical distribution for $X_{k}$ up to time $n$ is
\begin{equation*}
F_{k,\,T_{k}(n)}(x)\triangleq\frac{1}{|S_{k}(n)|}\sum_{t\in S_{k}(n)}I\left\{ x_{k,\,t}\leq x\right\},
\end{equation*}
where $|\cdot|$ denotes the number of items in a set. Similarly, from the sequence $\{x_{I_{t},\,t}\}_{t=1}^{n}$ drawn by the policy $\varphi$, we define
\begin{equation*}
F_{\varphi,\,n}(x)\triangleq\frac{1}{n}\sum_{t=1}^{n}I\left\{ x_{I_{t},\,t}\leq x\right\}.
\end{equation*}
The empirical estimate $\widehat{\rho}_{k,\,n}$ of the CVaR of arm $k$ may then be defined as
\begin{equation*}
\widehat{\rho}_{k,\,n}^{C}\triangleq F_{k,\,T_{k}(n)}^{-1}(\alpha)+(1-\alpha)^{-1}\frac{1}{T_{k}(n)}\sum_{t=1}^{T_{k}(n)}[x_{k,\,t}-F_{k,\,T_{k}(n)}^{-1}(\alpha)]_{+},
\end{equation*}
and the empirical estimate $\widehat{\rho}_{\varphi,\,n}$ of the CVaR for the policy $\varphi$ may be defined as
\begin{equation*}
\widehat{\rho}_{\varphi,\,n}^{C}\triangleq F_{\varphi,\,n}^{-1}(\alpha)+(1-\alpha)^{-1}\frac{1}{n}\sum_{t=1}^{n}[x_{I_{t},\,t}-F_{\varphi,\,n}^{-1}(\alpha)]_{+}.
\end{equation*}
According to (\ref{eq:pseudoregret}), the pseudo regret under the CVaR up to time $n$ is then
\begin{equation*}
R_{n}^{C}(\varphi)\triangleq \mathbf{E}\widehat{\rho}_{\varphi,\,n}^{C}-\rho_{k^{*}}^{C}.
\end{equation*}
For a fixed $\alpha$, let $\eta_{k}\triangleq F_{k}^{-1}(\alpha)$,
$\eta_{k,\,n}\triangleq F_{k,\,n}^{-1}(\alpha)$, $\eta_{\varphi,\,n}\triangleq F_{\varphi,\,n}^{-1}(\alpha)$,
$\eta^{*}\triangleq\eta_{k^{*}}$ and $\eta_{n}^{*}\triangleq\eta_{k^{*},\,n}$. Using these notations, we can represent the pseudo regret in the following explicit form
\begin{equation*}
R_{n}^{C}(\varphi)\triangleq\mathbf{E}(\eta_{\varphi,\,n}-\eta^{*})+(1-\alpha)^{-1}\frac{1}{n}\sum_{k=1}^{K}\mathbf{E}T_{k}(n)\mathbf{E}[(X_{k}-\eta_{\varphi,\,n})_{+}-(X_{k^{*}}-\eta^{*})_{+}].
\end{equation*}
The term $\varepsilon_{\rho}(n,\,T_k(n),\,K,\delta)$ for CVaR case is defined as follows
\begin{equation*}
\varepsilon_{\rho}^{C}(n,\,T_k(n),\,K,\delta) \triangleq \left[\left(1-\alpha\right)^{-1}\left(1-\frac{3\delta}{n}\right)M+2\left[1+\left(1-\alpha\right)^{-1}\right]m\left(\alpha\right)\right]\sqrt{\frac{\log\frac{2n^{2}K}{\delta}}{2T_{k}\left(n\right)}}
\end{equation*}
The term $\varepsilon_{\rho}^{C}$ consists of two parts. They are the corresponding confidence bounds on the two terms in $\widehat{\rho}_{k,\,n}^{C}$ respectively. Thus, the index for arm $k$ at time $n$ for CVaR is
\begin{equation*}
B_{k,\,n}^{C} \triangleq \widehat{\rho}_{k,\,n}^{C}- \varepsilon_{\rho}^{C}(n,\,T_k(n),\,K,\delta).
\end{equation*}
The upper bound for the pseudo regret under CVaR is established in the following theorem. We let
\begin{equation*}
M_{\varphi}\left(n\right)\triangleq\max\left\{ \alpha,\:1-\alpha\right\} O\left(\frac{\log n}{n}\right)m\left(\alpha\right)+2m\left(\alpha\right)\sqrt{\frac{\log\frac{4n}{\delta}}{2n}},
\end{equation*}
and
\begin{equation*}
M_{k}(n)\triangleq2\log\frac{2n^{2}K}{\delta}\left(\frac{2\left[1+\left(1-\alpha\right)^{-1}\right]m\left(\alpha\right)+\left(1-\alpha\right)^{-1}M}{\Delta_{k}^{C}-(1-\alpha)^{-1}\frac{4\delta}{n^{2}K}M}\right)^{2}+3\delta,
\end{equation*}
where $\Delta_{k}^{C}\triangleq \rho_{k}^{C}-\rho_{k^{*}}^{C}$.
\begin{theorem}\label{CVaR theorem}
\noindent For all $n>K$, the pseudo regret $R_{n}^{C}(\varphi)$ satisfies
\begin{align*}
R_{n}^{C}(\varphi) \leq
& (1-\frac{4\delta}{n})[M_{\varphi}(n)+(1-\alpha)^{-1}\frac{1}{n}\sum_{k=1}^{K}M_{k}(n)(M+\Delta_{k}^{C})]\\
& +\frac{4\delta}{n}[((1-\alpha)^{-1}+1)M+(1-\alpha)^{-1}\Delta_{k}^{C}].
\end{align*}
Specifically, the order of the upper bound is dominated by the term $M_{\varphi}(n)$ which is $O(\sqrt{\log n / n})$.
\end{theorem}

\subsection{Mean-deviation risk}\label{Mean-deviation risk measure}

We define the MD as follows.
\begin{definition}
\textbf{(MD risk measure)} The mean-deviation of arm $k$ is
\begin{equation*}
\rho_{k}^{M} \triangleq \mu_{k}+\gamma\Vert X_{k}\Vert_{p},
\end{equation*}
where $\Vert X_{k}\Vert _{p}\triangleq(\mathbf{E}|X_{k}-\mu_{k}|^{p})^{\frac{1}{p}}$,
$\gamma\geq0$ and $p\in\left[1,\:\infty\right)$.
\end{definition}
The constant $\gamma$ is a coefficient controlling the trade-off between the mean and the $L_{p}$ norm. When $\gamma=0$, the MD measure becomes expectation; in the special case of $p=2$, the second term becomes the standard deviation. Based on the observation sequence of arm $k$, the empirical estimate $\widehat{\rho}_{k,\,n}$ of the MD risk of arm $k$ may be defined as
\begin{equation*}
\widehat{\rho}_{k,\,n}^{M} \triangleq\overline{x}_{k,\,T_{k}(n)}+\gamma\left(\frac{1}{|S_{k}(n)|}\sum_{t\in S_{k}(n)}\left|x_{k,\,t}-\overline{x}_{k,\,T_{k}(n)}\right|^{p}\right)^{\frac{1}{p}},
\end{equation*}
where $\overline{x}_{k,\,T_{k}(n)}\triangleq\frac{1}{|S_{k}(n)|}\sum_{t\in S_{k}(n)}x_{k,\,t}$. Based on the whole observation sequence the empirical estimate $\widehat{\rho}_{\varphi,\,n}$ of the MD risk for the policy $\varphi$ may be defined as
\begin{equation*}
\widehat{\rho}_{\varphi,\,n}^{M} \triangleq \overline{x}_{\varphi,\,n}+\gamma\left(\frac{1}{n}\sum_{t=1}^{n}\left|x_{I_{t},\,t}-\overline{x}_{\varphi,\,n}\right|^{p}\right)^{\frac{1}{p}},
\end{equation*}
where $\overline{x}_{\varphi,\,n}=\frac{1}{n}\sum_{t=1}^{n}x_{I_{t},\,t}$.
The pseudo regret under the MD up to time $n$ is then
\begin{align*}
R_{n}^{M}(\varphi) & \triangleq\left[E\left(\overline{x}_{\varphi,\,n}\right)-\mu_{k^{*}}\right]+\gamma\left[\mathbf{E}\left(\frac{1}{n}\sum_{t=1}^{n}\left|x_{I_{t},\,t}-\overline{x}_{\varphi,\,n}\right|^{p}\right)^{\frac{1}{p}}-\left(\mathbf{E}\left|X_{k^{*}}-\mu_{k^{*}}\right|^{p}\right)^{\frac{1}{p}}\right].
\end{align*}
The term $\varepsilon_{\rho}(n,\,T_k(n),\,K,\delta)$ for MD is defined as below
\begin{equation*}
\varepsilon_{\rho}^{M}(n,\,T_k(n),\,K,\delta) \triangleq
M\sqrt{\frac{\log\frac{4n^{2}K}{\delta}}{T_{k}\left(n\right)}}-M\left[\left(p+1\right)\sqrt{\frac{\log\frac{4n^{2}K}{\delta}}{2T_{k}\left(n\right)}}\right]^{\frac{1}{p}}.
\end{equation*}
Similar to the CVaR case, $\varepsilon_{\rho}^{M}$ is also a summation of two terms. The are the corresponding bounds for the empirical mean and empirical $L_{p}$ norm in
$\widehat{\rho}_{k,\,n}^{M}$ respectively. So the index for arm $k$ at time $n$ for MD is
\begin{equation*}
B_{k,\,n}^{M} \triangleq \rho_{k,\,n}^{M}-\varepsilon_{\rho}^{M}(n,\,T_k(n),\,K,\delta).
\end{equation*}
The upper bound for the pseudo regret under the MD is given in the following theorem. We let
\begin{equation*}
N_{\varphi}\left(n\right)\triangleq\sum_{k\neq k^{*}}\left(1-\frac{\delta}{n}\right)m_{k}\log\frac{4n^{2}K}{\delta}+\left(k-1\right)\delta,
\end{equation*}
where $m_{k}$ will be specified in Lemma \ref{lemma 9} below.

\begin{theorem}\label{MD theorem}
\noindent For all $n>K$, the pseudo regret satisfies
\begin{equation*}
R_{n}^{M}(\varphi) \leq \left(\Delta_{k}^{M}+2\gamma pM^{p}\right)\frac{N_{\varphi}\left(n\right)}{n}+M\left(\frac{\delta}{n}+\left(1-\frac{\delta}{n}\right)\sqrt{\frac{\log\frac{4n^{2}K}{\delta}}{2\left(n-N_{\varphi}\left(n\right)\right)}}\right),
\end{equation*}
where $\Delta_{k}^{M} \triangleq \left|\mu_{k}-\mu_{k^{*}}\right|$. Specifically, the order of the upper bound is dominated by the term $\sqrt{\log(4n^{2}K/\delta)[2(n-N_{\varphi}(n))]^{-1}}$ which is of the order $O(\sqrt{\log n / n})$.
\end{theorem}

\subsection{Shortfall risk}\label{Shortfall risk measure}

We define shortfall risk as follows.
\begin{definition}
\textbf{(Shortfall risk measure)} Let $l:\:\mathbb{R}\rightarrow\mathbb{R}$
be a convex loss function, the shortfall risk measure of arm $k$ is defined as
\begin{equation*}
\rho_{k}^{S}\triangleq\inf\left\{ \kappa\in\mathbb{R}:\:\mathbf{E}\left[l\left(X_{k}-\kappa\right)\right]\leq0\right\}.
\end{equation*}
\end{definition}
When $l\left(t\right)=t$, shortfall risk becomes the expectation of $X_{k}$.
We define the empirical estimate $\widehat{\rho}_{k,\,n}$ of the shortfall of arm $k$ as
\begin{equation*}
\widehat{\rho}_{k,\,n}^{S} \triangleq \inf\left\{ \kappa\in \mathbb{R}:\:\frac{1}{|S_{k}(n)|}\sum_{t\in S_{k}(n)}l\left(x_{k,\,t}-\kappa\right)\leq0\right\},
\end{equation*}
and the empirical estimate $\widehat{\rho}_{\varphi,\,n}$ of the shortfall for the policy may be defined as
\begin{equation*}
\widehat{\rho}_{\varphi,\,n}^{S} \triangleq \inf\left\{ \kappa\in \mathbb{R}:\:\frac{1}{n}\sum_{t=1}^{n}l\left(x_{I_{t},\,t}-\kappa\right)\leq0\right\}.
\end{equation*}
The pseudo regret under the shortfall up to time $n$ is then
\begin{equation*}
R_{n}^{S}(\varphi)\triangleq\mathbf{E}\inf\left\{ \kappa\in \mathbb{R}:\:\frac{1}{n}\sum_{t=1}^{n}l\left(x_{I_{t},\,t}-\kappa\right)\leq0\right\} -SF_{k^{*}}.
\end{equation*}
We make the following assumption about the underlying loss function $l$.
\begin{assumption}\label{shortfall assumption 1}
The loss function $l\left(t\right)$ is continuous, strictly increasing
and Lipschitz of order 1 in $t$ in the closed interval of $\left[-M,\:M\right]$
with Lipschitz constant $C_{l}$. In particular, we can let $l\left(0\right)=0$
and $l\left(t\right)$ be bounded by a constant $M_{l}$ uniformly
in $\left[0,\:M\right]$. Meanwhile, the derivative of $l\left(t\right)$
exists in $[-M,\:M]$ (with one-sided derivative at the
boundary) and is uniformly lower bounded by a constant $m_{l}$, i.e.
$l^{'}(t)\geq m_{l}$ for all $t\in[-M,\:M]$.
\end{assumption}
Actually, Assumption \ref{shortfall assumption 1} ensures that the set $\{ \kappa\in\mathbb{R}:\:\mathbf{E}[l(X_{k}-\kappa)]\leq0\}$
is non-empty for any arm and that $\rho_{k}^{S}$ lies within $[0,\:M]$, $k=1,\ldots,K$
(as we will show in Lemma \ref{lemma 8} below). Assumption
\ref{shortfall assumption 1} holds for many common loss functions including $l\left(t\right)=t$
which is the special case of expectation and $l\left(t\right)=e^{t}-1$ which is the widely used exponential loss function. The term $\varepsilon_{\rho}(n,\,T_k(n),\,K,\delta)$ for shortfall is defined as follows
\begin{equation*}
\varepsilon_{\rho}^{S}(n,\,T_k(n),\,K,\delta) \triangleq
2M_{l}M_{G}\sqrt{\frac{\log\frac{4n^{2}K}{\delta}}{2T_{k}\left(n\right)}}.
\end{equation*}
where $M_{G}$ is a constant which will be specified in Lemma \ref{lemma 10}.
$\varepsilon_{\rho}^{S}$ is a direct confidence bound on the empirical shortfall risk of arm $k$. The index for arm $k$ for the shortfall case at time $n$ is then
\begin{equation*}
B_{k,\,n}^{S} \triangleq \widehat{\rho}_{k,\,n}^{S}-\varepsilon_{\rho}^{S}(n,\,T_k(n),\,K,\delta).
\end{equation*}
The upper bound for the pseudo regret under the shortfall risk is
given in the following theorem.
\begin{theorem}\label{shortfall theorem}
For all $n>K$, the pseudo regret $\widetilde{R}_{n}^{SF}(\varphi)$ satisfies
\begin{equation*}
R_{n}^{S}(\varphi)\leq(1-\frac{\delta}{n})(\sum_{k\neq k^{*}}\frac{8M_{l}^{2}M_{G}}{nm_{l}\Delta_{k}^{S}}\log\frac{4n^{2}K}{\delta}+2M_{l}M_{G}\sqrt{\frac{\log\frac{4n^{2}K}{\delta}}{2T_{k^{*}}(n)}})+\frac{\delta}{n}M,
\end{equation*}
where $\Delta_{k}^{S} \triangleq \rho_{k}^{S}-\rho_{k^{*}}^{S}$. Specifically, the order of the upper bound is dominated by the term $2M_{l}M_{G}\sqrt{\log(4n^{2}K/\delta)[2T_{k^{*}}(n)]^{-1}}$ which is of the order $O(\sqrt{\log n / n})$.
\end{theorem}

\subsection{Discussion}\label{Discussion}

For all three risk measures, we obtain a uniform upper bound of the order $O(\sqrt{\log n / n})$
on the pseudo regret. This convergence rate is different from the order of $O(\log n / n)$
in the risk-neutral case. In the risk-neutral MAB, the pseudo regret
$\sum_{k=1}^{K}\mathbf{E}T_{k}(n)(\mu_{k^{*}}-\mu_{k})$
is actually a linear function with respect to the average number of pulls from
the sub-optimal arms, and is thus dominated by the order of $\mathbf{E}T_{k}(n)$.
However, this situation does not hold for the risk-averse case. We explain in more detail for the specific cases of CVaR and MD. We reorganize the pseudo regret for CVaR and MD in a new consistent form, of which the derivation follows from the results in later sections. In this new form, the linear part and the non-linear part w.r.t $\mathbf{E}T_{k}(n)$ are separated for clarity. The form is presented as below

\begin{equation}\label{pseudo regret separate}
\widetilde{R}_{n}\left(\varphi\right)= \frac{1}{n}\sum_{k\neq k^{*}}\mathbf{E}T_{k}(n)S_{k}+S^{*}.
\end{equation}
In the case of the CVaR risk measure,
\begin{align*}
S_{k} & =\left(1-\alpha\right)^{-1}\mathbf{E}\left[\left(X_{k}-\eta_{\varphi,\,n}\right)_{+}-\left(X_{k^{*}}-\eta_{\varphi,\,n}\right)_{+}\right],\\
S^{*} & =\left(1-\alpha\right)^{-1}\mathbf{E}\left[\left(X_{k^{*}}-\eta_{\varphi,\,n}\right)_{+}-\left(X_{k^{*}}-\eta^{*}\right)_{+}\right]+\mathbf{E}\left(\eta_{\varphi,\,n}-\eta^{*}\right),
\end{align*}
when $\alpha=0$, all $\eta$’s become the essential infimum. Therefore,
\begin{align*}
S_{k} & =\mathbf{E}\left[\left(X_{k}-\eta_{\varphi,\,n}\right)-\left(X_{k^{*}}-\eta_{\varphi,\,n}\right)\right]=\mathbf{E}\left(X_{k}-X_{k^{*}}\right),\\
S^{*} & =\mathbf{E}\left[\left(X_{k^{*}}-\eta_{\varphi,\,n}\right)-\left(X_{k^{*}}-\eta^{*}\right)\right]+\mathbf{E}\left(\eta_{\varphi,\,n}-\eta^{*}\right)=0.
\end{align*}
In the case of the MD risk measure,
\begin{align*}
S_{k} & =\Delta_{k}+\gamma\left|\mathbf{E}\left|x_{k,\,t}-\overline{x}_{\varphi,\,n}\right|^{p}-\mathbf{E}\left|X_{k^{*}}-\mu_{k^{*}}\right|^{p}\right|,\\
S^{*} & =\frac{\gamma\mathbf{E}T_{k^{*}}\left(n\right)}{n}\left|\mathbf{E}\left|X_{k^{*}}-\overline{x}_{\varphi,\,n}\right|^{p}-\mathbf{E}\left|X_{k^{*}}-\mu_{k^{*}}\right|^{p}\right|,
\end{align*}
when $\gamma=0$,
\begin{align*}
S_{k} & =\Delta_{k}=\mathbf{E}\left(X_{k}-X_{k^{*}}\right),\\
S^{*} & =0.
\end{align*}
In the risk-averse case, the pseudo regret includes the nonlinear term of $S^{*}$. As each $S_{k}$ is bounded, the first term of (\ref{pseudo regret separate}) is dominated by the order of $\frac{1}{n}\sum_{k\neq k^{*}}\mathbf{E}T_{i}(n)$, which is of the order $O(\log n / n)$.
However, the second term is dominated by the asymptotic order of the specific statistic in each case. As a result, the overall order of the upper bound is actually $O(\sqrt{\log n / n})$.
In the special case where the three risk measures reduce to expectation, the second term $S^{*}$ is zero and the risk-averse pseudo regret is consistent with the risk-neutral pseudo regret. Thus, the asymptotic order of the upper bound on the pseudo regret becomes $O(\log n / n)$ for the risk-neutral case as expected.

\section{Proofs of the main results}\label{Proof of the main results}
In this section, we give the detailed proofs of our three main theorems. Throughout this section, we use the sequence $\{x_{k,\,t}\}_{t=1}^{n}$  which is an i.i.d. sequence with each $x_{k,\,t}$ subject to $F_{k}$. We introduce this sequence for notational convenience. Note that $\{x_{k,\,t}\}_{t=1}^{n}$ is not actually chosen by any algorithm, thus is unobservable.

The proofs in Sections \ref{CVaR risk measure proof}-\ref{Shortfall risk measure proof} all follow a similar general procedure. Recall that the index for each arm consists of two parts: $\widehat{\rho}_{k,\,n}$ and $\varepsilon_{\rho}(\cdot)$. We first prove that $\varepsilon_{\rho}(\cdot)$ is an upper bound on the estimation error $|\widehat{\rho}_{k,\,n}-\rho_{k}|$ with a probability of $1-O(n^{-2})$. Thus, the total probability that all $\rho_{k}$'s falls below the index at each of the first $n$ times is of order $O(n^{-1})$, which means a sub-optimal arm becomes less likely to be chosen by the algorithm as the number of pulls of it increases. Based on this conclusion, we may further show that the total number of pulls of all sub-optimal arms is upper bounded. This means that the portion of pulls of the optimal arm in the whole sequence increases, thus, the empirical risk based on this sequence is supposed to approximate the true risk of the optimal arm. Finally, we may prove the upper bound for the pseudo regret in each of our three main theorems. The proof of each starred lemma is detailed in the appendix.

\subsection{Conditional value-at-risk measure}\label{CVaR risk measure proof}

Recall the definitions of $\eta_{k}$, $\eta_{k,\,n}$, $\eta_{\varphi,\,n}$, $\eta^{*}$ and $\eta_{n}^{*}$ in Section \ref{CVaR risk measure}.
In this section, we first derive a probability upper bound for $|\eta_{k,\,n}-\eta_{k}|$ in Lemma \ref{lemma 2}. Based on the results in Lemma \ref{lemma 2}, we may establish an upper bound on $\widehat{\rho}_{k,\,n}^{C}$, then we may bound the number of pulls of all sub-optimal arms in Lemma \ref{lemma 4}. In Lemma \ref{lemma 5}, we construct the upper bound for another crucial term $|\eta_{\varphi,\,n}-\eta^{*}|$ that appears in the pseudo regret. Based on the results in Lemmas \ref{lemma 2}-\ref{lemma 5}, we may decompose the pseudo regret into different parts and bound each part. Then we prove the upper bound for the pseudo regret in Theorem \ref{CVaR theorem} in the end of this section.

\noindent \textbf{Step 1: Bounding} $|\eta_{k,\,n}-\eta_{k}|$ \textbf{which leads to a bound on} $|\widehat{\rho}_{k,\,n}^{C}-\rho_{k}^{C}|$.
\begin{lemma}\label{lemma 2}
\noindent Suppose Assumption \ref{CVaR assumption 1} holds, then for all $k=1,\ldots,K$, we have
\begin{equation}\label{lemma 2-1}
\eta_{k,\,n}\overset{\mathbf{P}}{\rightarrow}\eta_{k},
\end{equation}
as $n\rightarrow\infty$, and $\forall\:n>1$, with probability at least $1-\delta$
\begin{equation}\label{lemma 2-2}
|\eta_{k,\,n}-\eta_{k}|\leq\frac{2}{f_{k}(\eta_{k})}\sqrt{\frac{\log\frac{4}{\delta}}{2n}}.
\end{equation}
Furthermore, noting $M$ is the uniform essential bound on all $X_{k}$'s, we have
\begin{equation}\label{lemma 2-3}
\mathbf{E}\left(\left|\eta_{k,\,n}-\eta_{k}\right|\right)\leq\frac{2}{f_{k}(\eta_{k})}\sqrt{\frac{\log\frac{4}{\delta}}{2n}}
+\delta M.
\end{equation}
\end{lemma}
\begin{proof}
\noindent Note that $\eta_{k,\,n}$ is the $\alpha$-quantile of the empirical distribution
of arm $k$, i.e. the sample percentile of $\{x_{k,\,t}\}_{t=1}^{n}$.
From the Bahadur representation of sample percentiles {[}\citep{dasgupta2008asymptotic} Theorem 17.1{]}, we have
\begin{equation}\label{lemma 2-4}
\eta_{k,\,n}-\eta_{k}=\frac{1}{n}\sum_{t=1}^{n}\frac{\alpha-I\{x_{k,\,t}\leq \eta_{k}\}}{f_{k}(\eta_{k})}+r_{n},
\end{equation}
where $r_{n}=o_{p}(n^{-1/2})$. Since $\mathbf{E}[(\alpha-I\{x_{k,\,t}\leq \eta_{k}\})]=0$,
by the strong law of large numbers, $\eta_{k,\,n}-\eta_{k}\overset{\mathbf{P}}{\rightarrow}0$.
Meanwhile, $\alpha-I\{x_{k,\,t}\leq \eta_{k}\}\in[{\alpha-1},\:{\alpha}]$,
so by Hoeffding's inequality, we have
\begin{equation}\label{lemma 2-5}
\mathbf{P}\left\{ \left|\frac{1}{n}\sum_{t=1}^{n}\frac{\alpha-I\{x_{k,\,t}\leq \eta_{k}\}}{f_{k}(\eta_{k})}\right|\geq\frac{1}{f_{k}(\eta_{k})}\sqrt{\frac{\log\frac{4}{\delta}}{2n}}\right\} \leq\frac{\delta}{2}.
\end{equation}
Combining (\ref{lemma 2-4}) with (\ref{lemma 2-5}), we see that
\begin{align*}
 & \quad\mathbf{P}\left\{ |\eta_{k,\,n}-\eta_{k}|\geq\frac{2}{f_{k}(\eta_{k})}\sqrt{\frac{\log\frac{4}{\delta}}{2n}}\right\} \\
 & \leq \mathbf{P}\left\{ \left|\frac{1}{n}\sum_{t=1}^{n}\frac{\alpha-I\{x_{k,\,t}\leq\eta_{k}\}}{f_{k}(\eta_{k})}\right|\geq\frac{1}{f_{k}(\eta_{k})}\sqrt{\frac{\log\frac{4}{\delta}}{2n}}\right\} +\mathbf{P}\left\{ \left|r_{n}\right|\geq\frac{1}{f_{k}(\eta_{i})}\sqrt{\frac{\log\frac{4}{\delta}}{2n}}\right\} \\
 & =\left(1+o\left(1\right)\right)\mathbf{P}\left\{ \left|\frac{1}{n}\sum_{t=1}^{n}\frac{\alpha-I\{x_{k,\,t}\leq\eta_{k}\}}{f_{k}(\eta_{k})}\right|\geq\frac{1}{f_{k}(\eta_{k})}\sqrt{\frac{\log\frac{4}{\delta}}{2n}}\right\} \\
 & =\left(1+o\left(1\right)\right)\frac{\delta}{2}\leq\delta.
\end{align*}
Conclusion (\ref{lemma 2-2}) follows by noting that the term $o(1)$
in the inequality above is actually smaller than $1$.
Conclusion (\ref{lemma 2-3}) follows by noting that $|\eta_{k,\,n}-\eta_{k}|$ is bounded by $M$.
\hfill\rule{1.5mm}{3mm}
\end{proof}

\noindent\textbf{Step 2: Bounding the number of pulls of each sub-optimal arm}. To further analyze the upper bound, we define a key subset as follows
\begin{equation*}
\Omega_{n}^{C}\triangleq\left\{ \begin{array}{c}
\forall\:k=1,\ldots,K,\:\forall t=1,\ldots,n,\quad|\eta_{k,\,t}-\eta_{k}|\leq2m\left(\alpha\right)\sqrt{\frac{\log\frac{2n^{2}K}{\delta}}{2t}}\\
\:and\:\left|\frac{1}{t}\sum_{s=1}^{t}\left\{ (x_{k,\,s}-\eta_{k,\,t})_{+}-\mathbf{E}\left[(X_{k}-\eta_{k,\,t})_{+}\right]\right\} \right|\leq M\sqrt{\frac{\log\frac{2n^{2}K}{\delta}}{2t}}
\end{array}\right\},
\end{equation*}
where $m(\alpha)$ has been defined in Assumption \ref{CVaR assumption 1}. The subset $\Omega_{n}^{C}$ consists of all the realizations that each $\eta_{k,\,t}$ is close enough to its real value $\eta_{k}$ at each of the first $n$ times. Confined in $\Omega_{n}^{C}$, the probability that the index of any sub-optimal arm is larger than that of the optimal arm is controlled, with an lower bound given in the following lemma.
\begin{lemmastar}\label{lemma 3}
\noindent For all $n\geq 1$, we have $\mathbf{P}\{ \Omega_{n}^{C}\} \geq1-{3\delta}/{n}$.
\end{lemmastar}
Through the proofs of Lemma \ref{lemma 2} and \ref{lemma 3}, we have shown that $\varepsilon_{\rho}^{C}$ satisfies Condition 1. In the next lemma, we construct an upper bound on the number of pulls of each sub-optimal arm.

\begin{lemma}\label{lemma 4}
\noindent For any sub-optimal arm $k\neq k^{*}$ and for all $n\geq 1$, the following relation holds with probability at least $1-3\delta/n$
\begin{equation}\label{lemma 4-1}
T_{k}(n)\leq2\log\frac{2n^{2}K}{\delta}\left(\frac{2\left[1+\left(1-\alpha\right)^{-1}\right]m\left(\alpha\right)
+\left(1-\alpha\right)^{-1}M}{\Delta_{k}^{CVaR}-(1-\alpha)^{-1}\frac{4\delta}{n^{2}K}M}\right)^{2}.
\end{equation}
Furthermore,
\begin{equation}\label{lemma 4-2}
\mathbf{E}T_{k}(n)\leq\left(1-3\delta\right)2\log\frac{2n^{2}K}{\delta}\left(\frac{2\left[1+\left(1-\alpha\right)^{-1}\right]m\left(\alpha\right)
+\left(1-\alpha\right)^{-1}M}{\Delta_{k}^{C}-(1-\alpha)^{-1}\frac{4\delta}{n^{2}K}M}\right)^{2}+3\delta,
\end{equation}
\noindent specifically, $\mathbf{E}T_{k}(n)\leq O(\log n)$.
\end{lemma}
\begin{proof}
\noindent The main idea of the proof is to upper bound the difference between the empirical risk and the real risk of each arm. Further, constrained to $\Omega_{n}^{C}$, both an upper bound and a lower bound are available for the index of each arm. From this point, we may bound the number of pulls of each sub-optimal arm using the relation between the indices of the corresponding sub-optimal arm and the optimal arm. We first decompose $\widehat{\rho}_{k,\,n}^{C}$ as
\begin{align*}
\widehat{\rho}_{k,\,n}^{C} =
& \eta_{k,\,T_{k}(n)}+(1-\alpha)^{-1}\frac{1}{T_{k}(n)}\sum_{t=1}^{T_{k}(n)}(x_{k,\,t}-\eta_{k,\,T_{k}(n)})_{+}\\
= & \rho_{k}^{C}+(\eta_{k,\,T_{k}(n)}-\eta_{k})\\
& +(1-\alpha)^{-1}\left\{ \frac{1}{T_{k}\left(n\right)}\sum_{t=1}^{T_{k}\left(n\right)}(x_{k,\,t}-\eta_{k,\,T_{k}\left(n\right)})_{+}-\mathbf{E}\left[(X_{k}-\eta_{k,\,T_{k}(n)})_{+}\right]\right\} \\
 & +(1-\alpha)^{-1}\left\{ \mathbf{E}\left[(X_{k}-\eta_{k,\,T_{k}(n)})_{+}\right]-\mathbf{E}\left[(X_{k}-\eta_{k})_{+}\right]\right\} .
\end{align*}
We will bound the last three terms in the right side of the equation above. The last term in the above equation is bounded below by
\begin{align}\label{lemma 4-3}
& -(1-\alpha)^{-1}\mathbf{E}\left(\left|\eta_{k,\,T_{k}\left(n\right)}-\eta_{k}\right|\right) \nonumber \\
\leq & (1-\alpha)^{-1}\left\{ \mathbf{E}\left[(X_{k}-\eta_{k,\,T_{k}(n)})_{+}\right]-\mathbf{E}\left[(X_{k}-\eta_{k})_{+}\right]\right\} \nonumber \\
\leq & (1-\alpha)^{-1}\mathbf{E}\left(\left|\eta_{k,\,T_{k}\left(n\right)}-\eta_{k}\right|\right).
\end{align}
Let ${2\delta}/{n^{2}K}$ replace $\delta$ in Lemma \ref{lemma 2}, we further obtain the following lower and upper bounds
\begin{align}\label{lemma 4-4}
& -(1-\alpha)^{-1}\left(2M\left(\alpha\right)\sqrt{\frac{\log\frac{2n^{2}K}{\delta}}{2T_{k}\left(n\right)}}
+\frac{2\delta}{n^{2}K}M\right) \nonumber \\
\leq & (1-\alpha)^{-1}\left\{ \mathbf{E}\left[(X_{k}-\eta_{k,\,T_{k}(n)})_{+}\right]-\mathbf{E}\left[(X_{k}-\eta_{k})_{+}\right]\right\} \nonumber \\
\leq & (1-\alpha)^{-1}\left(2m\left(\alpha\right)\sqrt{\frac{\log\frac{2n^{2}K}{\delta}}{2T_{k}\left(n\right)}}
+\frac{2\delta}{n^{2}K}M\right).
\end{align}
Meanwhile, restricted to the subset $\Omega_{n}^{C}$ and for $t=T_{k,\,n}$, it is true that
\begin{equation}\label{lemma 4-5}
-2m(\alpha)\sqrt{\frac{\log\frac{2n^{2}K}{\delta}}{2T_{k}(n)}}\leq\eta_{k,\,T_{k}(n)}
-\eta_{k}\leq2m(\alpha)\sqrt{\frac{\log\frac{2n^{2}K}{\delta}}{2T_{k}(n)}},
\end{equation}
and
\begin{align}\label{lemma 4-6}
-(1-\alpha)^{-1}M\sqrt{\frac{\log\frac{2n^{2}K}{\delta}}{2T_{k}(n)}} & \leq(1-\alpha)^{-1}\left\{ \frac{1}{T_{k}(n)}\sum_{t=1}^{T_{k}(n)}(x_{k,\,t}-\eta_{k,\,T_{k}(n)})_{+}-\mathbf{E}\left[(X_{k}-\eta_{k,\,T_{k}(n)})_{+}\right]\right\} \nonumber \\
& \leq(1-\alpha)^{-1}M\sqrt{\frac{\log\frac{2n^{2}K}{\delta}}{2T_{k}(n)}}.
\end{align}
Combining (\ref{lemma 4-3})-(\ref{lemma 4-6}), we see that $\widehat{\rho}_{k,\,n}^{C}$ is bounded. Then we conclude that the following relation holds for the index for any arm $k$, $k=1,\ldots,K$ with probability at least $(1-3\delta/n)$, in particular, for the optimal arm $k^{*}$
\begin{align*}
\rho_{k}^{C}
& -2\left(2\left[1+\left(1-\alpha\right)^{-1}\right]m\left(\alpha\right)+\left(1-\alpha\right)^{-1}M\right)\sqrt{\frac{\log\frac{2n^{2}K}{\delta}}{2T_{k}(n)}}-(1-\alpha)^{-1}\frac{2\delta}{n^{2}K}M \\
& \leq B_{k,\,n}^{C} \leq \rho_{k}^{C}+(1-\alpha)^{-1}\frac{2\delta}{n^{2}K}M.
\end{align*}
\noindent At any time $n+1$, if arm $k\neq k^{*}$ is pulled, i.e. $I_{n}=k$, then the index of arm $k$ must be lower than that of any other arm, including the optimal arm $k^{*}$. So we must have
\begin{align*}
B_{k,n}^{C}
& =\widehat{\rho}_{k,\,n}^{C}-(1-\alpha)^{-1}\left(1-\frac{3\delta}{n}\right)M\sqrt{\frac{\log\frac{2n^{2}K}{\delta}}{2T_{k}(n)}} \\
& \leq \widehat{\rho}_{k^{*},\,n}^{C}-(1-\alpha)^{-1}\left(1-\frac{3\delta}{n}\right)M\sqrt{\frac{\log\frac{2n^{2}K}{\delta}}{2T_{k}(n)}}=B_{k^{*},\,n}^{C},
\end{align*}
which means, with probability at least $(1-3\delta/n)$, the following relation holds
\begin{align}\label{lemma 4-7}
\rho_{k}^{C}
& -2\left(2\left[1+\left(1-\alpha\right)^{-1}\right]m\left(\alpha\right)+\left(1-\alpha\right)^{-1}M\right)\sqrt{\frac{\log\frac{2n^{2}K}{\delta}}{2T_{k}\left(n\right)}} \nonumber \\
& -(1-\alpha)^{-1}\frac{2\delta}{n^{2}K}M \leq \rho_{k^{*}}^{C}+(1-\alpha)^{-1}\frac{2\delta}{n^{2}K}M.
\end{align}
Relation (\ref{lemma 4-7}) directly leads to conclusion (\ref{lemma 4-1}). Conclusion
(\ref{lemma 4-2}) follows by noting that $T_{k}(n)$ is automatically bounded by $n$.
\hfill\rule{1.5mm}{3mm}
\end{proof}

\noindent \textbf{Step 3: Establishing an upper bound on the pseudo regret}. The convergence of the sequence
$\{ \eta_{\varphi,\,n}\} _{n=1}^{\infty}$ is necessary for our derivation of the upper bound on the pseudo regret. Under general
conditions, there is no guarantee of the convergence of the sequence $\eta_{\varphi,\,n}$ as $n\rightarrow\infty$.
However, if we let $k_{n}\triangleq\sum_{k\neq k^{*}}T_{k}(n)$ be the total number of pulls from sub-optimal arms, then we see that $k_{n}=O\left(\log n\right)$. Based on this fact, the following lemma establishes the convergence of the sequence
$\{\eta_{\varphi,\,n}\} _{n=1}^{\infty}$.

\begin{lemma}\label{lemma 5}
If $k_{n}$ satisfies $k_{n}=O\left(\log n\right)$, then
\begin{equation}\label{lemma 5-1}
\eta_{\varphi,\,n}\overset{\mathbf{P}}{\rightarrow}\eta^{*},
\end{equation}
\noindent as $n\rightarrow\infty$, and $\forall\:n\geq 1$, with
probability at least $1-\delta$ we have
\begin{equation}\label{lemma 5-2}
\left|\eta_{\varphi,\,n}-\eta^{*}\right|\leq\max\left\{ \alpha,\:1-\alpha\right\} O\left(\frac{\log n}{n}\right)m\left(\alpha\right)+2m\left(\alpha\right)\sqrt{\frac{\log\frac{4}{\delta}}{2n}}.
\end{equation}
\end{lemma}
\begin{proof}
\noindent First we separate the observation sequence up to time $n$ into two parts: $k_{n}$ observations from the sub-optimal arms and
$(n-k_{n})$ observations from the optimal arm. By considering the extreme cases where all the observations in the latter category are
larger or smaller than those in the former, we conclude that
\begin{equation*}
F_{k^{*},\,\left(n-k_{n}\right)}^{-1}\left[\left(1+\frac{k_{n}}{n}\right)\alpha-\frac{k_{n}}{n}\right]\leq\eta_{\varphi,\,n}\leq F_{k^{*},\,\left(n-k_{n}\right)}^{-1}\left[\left(1+\frac{k_{n}}{n}\right)\alpha\right].
\end{equation*}
Let $\alpha_{1}=(1+k_{n}/{n})\alpha-k_{n}/{n}$
and $\alpha_{2}=(1+k_{n}/{n})\alpha$. Based on Lemma
\ref{lemma 2}, for all $n\geq 1$, the following bounds on $\eta_{\varphi,\,n}$ hold with probability
at least $1-\delta$
\begin{align*}
 & F_{k^{*}}^{-1}(\alpha_{1})-\frac{2}{f_{k^{*}}\left(F_{k^{*}}^{-1}(\alpha_{1})\right)}\sqrt{\frac{\log\frac{4}{\delta}}{2n}}\leq\eta_{\varphi,\,n}\leq F_{k^{*}}^{-1}(\alpha_{2})+\frac{2}{f_{k^{*}}\left(F_{k^{*}}^{-1}(\alpha_{2})\right)}\sqrt{\frac{\log\frac{4}{\delta}}{2n}}.
\end{align*}
Considering $\eta^{*}=F_{k^{*}}^{-1}(\alpha)$, the following bound on $|\eta_{\varphi,\,n}-\eta^{*}|$ holds with probability $1-\delta$
\begin{equation*}
\left|\eta_{\varphi,\,n}-\eta^{*}\right|\leq\max\left\{ \begin{array}{c}
F_{k^{*}}^{-1}\left(\alpha\right)-F_{k^{*}}^{-1}\left(\alpha_{1}\right)+\frac{2}{f_{k^{*}}\left(F_{k^{*}}^{-1}(\alpha_{1})\right)}\sqrt{\frac{\log\frac{4}{\delta}}{2n}},\\
F_{k^{*}}^{-1}\left(\alpha_{2}\right)-F_{k^{*}}^{-1}\left(\alpha\right)+\frac{2}{f_{k^{*}}\left(F_{k^{*}}^{-1}(\alpha_{2})\right)}\sqrt{\frac{\log\frac{4}{\delta}}{2n}}
\end{array}\right\}.
\end{equation*}
Under Assumption \ref{CVaR assumption 1}, from the mean value theorem and the condition $k_{n}=O(\log n)$,
as $n\rightarrow\infty$ we have
\begin{equation*}
F_{k^{*}}^{-1}(\alpha)-F_{k^{*}}^{-1}\left(\alpha_{1}\right)=O(\frac{k_{n}}{n})(1-\alpha)\frac{1}{f_{k}^{*}(\eta_{k})}\leq O(\frac{k_{n}}{n})(1-\alpha)m(\alpha),
\end{equation*}
and
\begin{equation*}
F_{k^{*}}^{-1}(\alpha_{2})-F_{k^{*}}^{-1}(\alpha)\leq O\left(\frac{\log n}{n}\right)\alpha m\left(\alpha\right).
\end{equation*}
The desired conclusion then follows.
\hfill\rule{1.5mm}{3mm}
\end{proof}
With upper confidence bounds on both $|\eta_{\varphi,\,n}-\eta^{*}|$ and $T_{k}\left(n\right)$,
we can now derive the upper bound on the pseudo regret $R_{n}^{C}(\varphi)$ in Theorem \ref{CVaR theorem}.

\begin{proof}[\textup{\textbf{Proof of Theorem \ref{CVaR theorem}}}]
Based on Lemma \ref{lemma 3}, restricted to the set $\Omega_{n}^{C}$, $\left|\eta_{\varphi,\,n}-\eta^{*}\right|$ can be upper bounded by $M_{\varphi}(n)$ with probability at least $1-{\delta}/{n}$.
Define the subset by $\varLambda_{n}\subseteq\Omega_{n}$ where $|\eta_{\varphi,\,n}-\eta^{*}|$ is bounded, then $\mathbf{P}\{ \varLambda_{n}\} \geq1-4\delta/n$.
So with probability at least $(1-{\delta}/{n})\mathbf{P}\{ \Omega_{n}^{C}\} =(1-{\delta}/{n})(1-{3\delta}/{n})\geq1-{4\delta}/{n}$,
$|\eta_{\varphi,\,n}-\eta^{*}|$ can be upper bounded by $M_{\varphi}(n)$. We first decompose and bound the pseudo regret with
\begin{align}\label{CVaR theorem 1}
R_{n}^{C}(\varphi) \leq &\mathbf{E}(|\eta_{\varphi,\,n}-\eta^{*}||\varLambda_{n})+\mathbf{E}(|\eta_{\varphi,\,n}-\eta^{*}||\varLambda_{n}^{c})\nonumber\\
& +(1-\alpha)^{-1}\frac{1}{n}\sum_{k=1}^{K}\mathbf{E}T_{k}(n)\mathbf{E}[(X_{k}-\eta_{\varphi,\,n})_{+}-(X_{k^{*}}-\eta^{*})_{+}|\varLambda_{n}]\nonumber\\
& +(1-\alpha)^{-1}\frac{1}{n}\sum_{k=1}^{K}\mathbf{E}T_{k}(n)\mathbf{E}[(X_{k}-\eta_{\varphi,\,n})_{+}-(X_{k^{*}}-\eta^{*})_{+}|\varLambda_{n}^{c}].
\end{align}
Note that
\begin{align*}
\mathbf{E}\left[(X_{k}-\eta_{\varphi,\,n})_{+}-(X_{k^{*}}-\eta^{*})_{+}\right]
& \leq\mathbf{E}\left[(X_{k}-\eta_{k})_{+}\right]-\mathbf{E}\left[(X_{k^{*}}-\eta^{*})_{+}\right] + \left|\eta_{\varphi n}-\eta_{k}\right| \\
& \leq\Delta_{k}^{C}+M.
\end{align*}
Meanwhile, since $\left|\eta_{\varphi,\,n}-\eta^{*}\right|$ is naturally bounded by $M$ and $\mathbf{E}T_{k}(n)$ is bounded by $M_{k}(n)$,
each term in (\ref{CVaR theorem 1}) can now be bounded. Then, we have
\begin{align*}
R_{n}^{C}(\varphi) \leq
& (1-\frac{4\delta}{n})M_{\varphi}(n)+\frac{4\delta}{n}M +(1-\alpha)^{-1}(1-\frac{4\delta}{n})\frac{1}{n}\sum_{k\neq k^{*}}M_{k}(n)(\Delta_{k}^{C}+M) \\
& +(1-\alpha)^{-1}\frac{4\delta}{n}(\Delta_{k}^{C}+M).
\end{align*}
\noindent The conclusion then follows from combining the terms in the above inequality together.
\hfill\rule{1.5mm}{3mm}
\end{proof}

\subsection{Mean-deviation risk measure}\label{Mean-deviation risk measure proof}

In this section, we first give some preliminary results in Lemma \ref{lemma 6} and Lemma \ref{lemma 8}. Then in Lemma \ref{lemma 9}, we derive the upper bound on the number of pulls of each sub-optimal arm. In the end of the section, we give the proof of Theorem \ref{MD theorem}.

\noindent\textbf{Step 1: Preliminary results}. We begin with the following lemma.
\begin{lemma}\label{lemma 6}
\noindent Let $f(x)=x^{p}$ and $g(x)=x^{\frac{1}{p}}$ be defined
on $[0,\:+\infty)$, where $p\geq1$. Then, for any $x,\:y\in[0,\:+\infty)$, we have
\begin{align*}
\left|f\left(x\right)-f\left(y\right)\right| & \leq f^{'}\left(\max\left\{ x,\,y\right\} \right)\left|x-y\right|,\\
\left|g\left(x\right)-g\left(y\right)\right| & \leq g\left(\left|x-y\right|\right),
\end{align*}
\noindent specifically, if $x,\:y\in\left[0,\:M\right]$, then \textup{$\left|f\left(x\right)-f\left(y\right)\right|\leq pM^{p-1}\left|x-y\right|$.}
\end{lemma}
The conclusion in Lemma \ref{lemma 6} follows directly from the convexity of $f(x)$ and the concavity
of $g(x)$. The proof is trivial so we omit it. We also need the preliminary result in the following lemma.
\begin{lemmastar}\label{lemma 8}
For $x\geq0$, if
\begin{equation*}
ax+bx^{\frac{1}{p}}\geq c,
\end{equation*}
\noindent where $a$, $b$ and $c$ are fixed constants, $p\geq1$, then we can conclude that
\begin{equation*}
x\geq\min\left\{ 1,\:\frac{c}{a+b},\:\left(\frac{c}{a+b}\right)^{p}\right\} .
\end{equation*}
\end{lemmastar}

\noindent\textbf{Step 2: Bounding the number of pulls of each sub-optimal arm}. For further discussion, we define the following key subset
\begin{equation*}
\Omega_{n}^{M}\triangleq\left\{ \begin{array}{c}
\forall\:i=k,\ldots,K,\;\forall t=1,\ldots,n,\:s.t.\:|\overline{x}_{k,\,t}-\mu_{k}|\leq M\sqrt{\frac{\log\frac{4n^{2}K}{\delta}}{2t}}\\
\;and\;\left|\frac{1}{t}\sum_{s=1}^{t}\left|x_{k,\,s}-\mu_{k}\right|^{p}-\mathbf{E}\left|X_{k}-\mu_{k}\right|^{p}\right|\leq M^{p}\sqrt{\frac{\log\frac{4n^{2}K}{\delta}}{2t}}
\end{array}\right\}.
\end{equation*}
The subset $\Omega_{n}^{M}$ has the same meaning as $\Omega_{n}^{C}$, which represents the set of realizations where the empirical risk of each arm is close enough to the real risk at each of the first $n$ times. By establishing the lower bound of the probability of $\Omega_{n}^{M}$ in the following lemma, we may show that the
$\widehat{\rho}_{k,\,n}^{M}$ approximates $\rho_{k}^{M}$ with a high probability.
\begin{lemmastar}\label{lemma 7}
For all $n\geq 1$, we have $\mathbf{P}\{\Omega_{n}^{M}\}\geq1-{\delta}/{n}$.
\end{lemmastar}
\noindent From the proof of Lemma \ref{lemma 7}, we have shown that
$\varepsilon_{\rho}^{M}$ satisfies Condition 1. In the following lemma, we bound the number of pulls from the sub-optimal arms. We let
\begin{equation*}
m_{k}\triangleq(\min\{ 1,\:\frac{\rho_{k}^{M}-\rho_{k^{*}}^{M}}{2M[1+(1+p)^{1/p}]},\:\frac{(\rho_{k}^{M}-\rho_{k^{*}}^{M})^{p}}{2^{p}M^{p}[1+(1+p)^{1/p}]^{p}}\} )^{-2},
\end{equation*}
for all $k=1,\ldots,K$.

\begin{lemma}\label{lemma 9}
\noindent For any sub-optimal arm $k\neq k^{*}$, the following relation holds with probability at least $1-{\delta}/{n}$
\begin{equation}\label{lemma 9-1}
T_{k}(n)\leq m_{k}\log\frac{4n^{2}K}{\delta}.
\end{equation}
Furthermore,
\begin{equation}\label{lemma 9-2}
\mathbf{E}T_{k}(n)\leq\left(1-\frac{\delta}{n}\right)m_{k}\log\frac{4n^{2}K}{\delta}+\delta.
\end{equation}
Specifically, we have $\mathbf{E}T_{i}(n)\leq O\left(\log n\right)$.
\end{lemma}
\begin{proof}
The main idea of the proof is the same as Lemma \ref{lemma 4}. Our task is to establish the upper and lower bounds for the index of arm $k$ based on the bound of $|\widehat{\rho}_{k,\,n}^{M}-\rho_{k}^{M}|$, for all $k=1,\ldots,K$. The empirical estimate $\widehat{\rho}_{k,\,n}^{M}$ consists of two terms, we first bound these two terms in this proof. Restricted to the set $\Omega_{n}^{M}$, we have
\begin{equation*}
-M\sqrt{\frac{\log\frac{4n^{2}K}{\delta}}{2T_{k}(n)}}\leq\overline{x}_{k,\,T_{k}(n)}-\mu_{k}\leq M\sqrt{\frac{\log\frac{4n^{2}K}{\delta}}{2T_{k}(n)}}.
\end{equation*}
So we see that the first term of $\widehat{\rho}_{k,\,n}^{M}$ is bounded. From Lemma \ref{lemma 6}, we have
\begin{equation*}
\left|\left(\frac{1}{T_{k}(n)}\sum_{t=1}^{T_{k}(n)}\left|x_{k,\,t}-\overline{x}_{k,\,T_{k}(n)}\right|^{p}\right)^{\frac{1}{p}}-\left\Vert X_{k}\right\Vert _{p}\right|\leq\left|\frac{1}{T_{k}(n)}\sum_{t=1}^{T_{k}(n)}\left|x_{k,\,t}-\overline{x}_{k,\,T_{i}(n)}\right|^{p}-\mathbf{E}\left|X_{k}-\mu_{k}\right|^{p}\right|^{\frac{1}{p}}.
\end{equation*}
The right side of the inequality above is further bounded by
\begin{align*}
& \left|\frac{1}{T_{k}(n)}\sum_{t=1}^{T_{k}(n)}\left|x_{k,\,t}-\overline{x}_{k,\,T_{k}(n)}\right|^{p}-\mathbf{E}\left|X_{k}-\mu_{k}\right|^{p}\right|\\ \leq & \left|\frac{1}{T_{k}(n)}\sum_{t=1}^{T_{k}(n)}\left|x_{k,\,t}-\overline{x}_{k,\,T_{k}(n)}\right|^{p}-\frac{1}{T_{k}(n)}\sum_{t=1}^{T_{k}(n)}\left|x_{k,\,t}-\mu_{k}\right|^{p}\right|\\
& +\left|\frac{1}{T_{k}(n)}\sum_{t=1}^{T_{k}(n)}\left|x_{k,\,t}-\mu_{k}\right|^{p}-\mathbf{E}\left|X_{k}-\mu_{k}\right|^{p}\right|\\
\leq & \frac{1}{T_{k}(n)}\sum_{t=1}^{T_{k}(n)}\left|\left|x_{k,\,t}-\overline{x}_{k,\,T_{k}(n)}\right|^{p}-\left|x_{k,\,t}-\mu_{k}\right|^{p}\right|
+M^{p}\sqrt{\frac{\log\frac{4n^{2}K}{\delta}}{2T_{k}(n)}}\\
\leq & \frac{1}{T_{k}(n)}\sum_{t=1}^{T_{k}(n)}pM^{p-1}\left|\left|x_{k,\,t}-\overline{x}_{k,\,T_{k}(n)}\right|-\left|x_{k,\,t}-\mu_{k}\right|\right|
 +M^{p}\sqrt{\frac{\log\frac{4n^{2}K}{\delta}}{2T_{k}(n)}}\\
 \leq & \frac{1}{T_{k}(n)}\sum_{t=1}^{T_{k}(n)}pM^{p-1}\left|\overline{x}_{k,\,T_{k}(n)}-\mu_{k}\right|+M^{p}\sqrt{\frac{\log\frac{4n^{2}K}{\delta}}{2T_{k}(n)}}
 \leq \left(p+1\right)M^{p}\sqrt{\frac{\log\frac{4n^{2}K}{\delta}}{2T_{k}(n)}},
\end{align*}
where the third inequality follows from the conclusion in Lemma \ref{lemma 5} and the last
inequality follows from the fact that $|\overline{x}_{k,\,T_{k}(n)}-\mu_{k}|\leq M$. So, we have
\begin{equation*}
-M[(p+1)\sqrt{\frac{\log\frac{4n^{2}K}{\delta}}{2T_{k}(n)}}]^{\frac{1}{p}}\leq(\frac{1}{T_{k}(n)}\sum_{t=1}^{T_{k}(n)}|x_{k,\,t}-\overline{x}_{k,\,T_{k}(n)}|^{p})^{\frac{1}{p}}-\Vert X_{k}\Vert _{p}\leq M[(p+1)\sqrt{\frac{\log\frac{4n^{2}K}{\delta}}{2T_{k}(n)}}]^{\frac{1}{p}}.
\end{equation*}
From the three relations above, we see that the second term of $\widehat{\rho}_{k,\,n}^{M}$ is also bounded. Suppose that arm $k\neq k^{*}$ is pulled at time $n+1$ by Algorithm 1, then the following relation must hold
\begin{align*}
\rho_{k}^{M}-2M\sqrt{\frac{\log\frac{4n^{2}K}{\delta}}{T_{k}\left(n\right)}}-2M\left[\left(p+1\right)\sqrt{\frac{\log\frac{4n^{2}K}{\delta}}{2T_{k}\left(n\right)}}\right]^{\frac{1}{p}} & \leq B_{k^{*},\,n}^{M}\leq B_{k,\,n}^{M} \leq \rho_{k^{*}}^{M},
\end{align*}
and so
\begin{equation*}
\rho_{k}^{M}-\rho_{k^{*}}^{M}
\leq 2M\sqrt{\frac{\log\frac{4n^{2}K}{\delta}}{T_{k}\left(n\right)}}+2M\left[\left(p+1\right)\sqrt{\frac{\log\frac{4n^{2}K}{\delta}}{2T_{k}\left(n\right)}}\right]^{\frac{1}{p}}.
\end{equation*}
As in the proof of Lemma \ref{lemma 4}, conclusion (\ref{lemma 9-1}) follows by substituting $x=\sqrt{{T_{k}^{-1}(n)}{\log({4n^{2}K}/{\delta})}}$, $a=2M$, $b=2M\left(p+1\right)^{\frac{1}{p}}$ and $c=\rho_{k}^{M}-\rho_{k^{*}}^{M}$
in Lemma \ref{lemma 8}. Conclusion (\ref{lemma 9-2}) directly follows by noting $T_{k}(n)$
is automatically bounded by $n$.
\hfill\rule{1.5mm}{3mm}
\end{proof}

\noindent\textbf{Step 3: Establishing an upper bound on the pseudo regret}. Note that the pseudo regret consists of two parts: the differences between means and deviations. Based on the previous two steps, we can bound each of these two parts. The derivation of the upper bound for the pseudo regret $R_{n}^{M}(\varphi)$ in Theorem 2 is given below.
\begin{proof}[\textup{\textbf{Proof of Theorem \ref{MD theorem}}}]
We first decompose the pseudo regret as follows
\begin{align}\label{MD theorem 1}
R_{n}^{M}(\varphi)
& =\left[E\left(\overline{x}_{\varphi,\,n}\right)-\mu_{k^{*}}\right]+\gamma\left[\mathbf{E}\left(\frac{1}{n}\sum_{t=1}^{n}\left|x_{I_{t},\,t}-\overline{x}_{\varphi,\,n}\right|^{p}\right)^{\frac{1}{p}}-\left(\mathbf{E}\left|X_{k^{*}}-\mu_{k^{*}}\right|^{p}\right)^{\frac{1}{p}}\right]\nonumber \\
\leq & \frac{1}{n}\sum_{k\neq k^{*}}\mathbf{E}T_{k}(n)\Delta_{k}^{MD}+\gamma\left|\left(\frac{1}{n}\sum_{t=1}^{n}\mathbf{E}\left|x_{I_{t},\,t}-\overline{x}_{\varphi,\,n}\right|^{p}\right)^{\frac{1}{p}}-\left(\mathbf{E}\left|X_{k^{*}}-\mu_{k^{*}}\right|^{p}\right)^{\frac{1}{p}}\right|\nonumber \\
\leq & \frac{1}{n}\sum_{k\neq k^{*}}\mathbf{E}T_{k}(n)\Delta_{k}^{MD}+\frac{\gamma}{n}\sum_{t=1}^{n}\left|\mathbf{E}\left|x_{I_{t},\,t}-\overline{x}_{\varphi,\,n}\right|^{p}-\mathbf{E}\left|X_{k^{*}}-\mu_{k^{*}}\right|^{p}\right|\nonumber \\
\leq & \frac{1}{n}\sum_{k\neq k^{*}}\mathbf{E}T_{k}(n)\Delta_{k}^{MD}+\frac{\gamma}{n}\sum_{k=1}^{K}\mathbf{E}T_{k}(n)\left|\mathbf{E}\left|x_{k,\,t}-\overline{x}_{\varphi,\,n}\right|^{p}-\mathbf{E}\left|X_{k^{*}}-\mu_{k^{*}}\right|^{p}\right|\nonumber \\
= & \frac{1}{n}\sum_{k\neq k^{*}}\mathbf{E}T_{k}(n)\Delta_{k}^{MD}+\frac{\gamma}{n}\sum_{k\neq k^{*}}\mathbf{E}T_{k}(n)\left|\mathbf{E}\left|x_{k,\,t}-\overline{x}_{\varphi,\,n}\right|^{p}-\mathbf{E}\left|X_{k^{*}}-\mu_{k^{*}}\right|^{p}\right|\nonumber \\
& +\frac{\gamma\mathbf{E}T_{k^{*}}\left(n\right)}{n}\left|\mathbf{E}\left|X_{k^{*}}-\overline{x}_{\varphi,\,n}\right|^{p}-\mathbf{E}\left|X_{k^{*}}-\mu_{k^{*}}\right|^{p}\right|,
\end{align}
where the second inequality comes from Lemma \ref{lemma 6}. Next, we bound the last two terms in (\ref{MD theorem 1}). The second term in (\ref{MD theorem 1}) is bounded by
\begin{align*}
\frac{\gamma}{n}\sum_{k\neq k^{*}}\mathbf{E}T_{k}(n)\left|\mathbf{E}\left|x_{k,\,t}-\overline{x}_{\varphi,\,n}\right|^{p}-\mathbf{E}\left|X_{k^{*}}-\mu_{k^{*}}\right|^{p}\right| & \leq\frac{\gamma pM^{p}}{n}\sum_{k\neq k^{*}}\mathbf{E}T_{k}(n),
\end{align*}
due to Lemma \ref{lemma 6}. The third term in (\ref{MD theorem 1}) is bounded by
\begin{align*}
 & \frac{\gamma\mathbf{E}T_{k^{*}}\left(n\right)}{n}\left|\mathbf{E}\left|X_{k^{*}}-\overline{x}_{\varphi,\,n}\right|^{p}-\mathbf{E}\left|X_{k^{*}}-\mu_{k^{*}}\right|^{p}\right|\\
\leq & \gamma\left|\mathbf{E}\left|X_{k^{*}}-\overline{x}_{\varphi,\,n}\right|^{p}-\mathbf{E}\left|X_{k^{*}}-\mu_{k^{*}}\right|^{p}\right|\\
\leq & \gamma\mathbf{E}\left|\left|X_{k^{*}}-\overline{x}_{\varphi,\,n}\right|^{p}-\left|X_{k^{*}}-\mu_{k^{*}}\right|^{p}\right|\\
\leq & \gamma pM^{p-1}\mathbf{E}\left|\overline{x}_{\varphi,\,n}-\mu_{k^{*}}\right|\\
\leq & \gamma pM^{p-1}\mathbf{E}\left|\frac{1}{n}\sum_{k\neq k^{*}}\sum_{t=1}^{T_{k}\left(n\right)}\left(x_{k,\,t}-\mu_{k^{*}}\right)\right|+\gamma pM^{p-1}\mathbf{E}\left|\frac{1}{n}\sum_{t=1}^{T_{k^{*}}\left(n\right)}\left(x_{k^{*},\,t}-\mu_{k^{*}}\right)\right|\\
\leq & \gamma pM^{p-1}\frac{1}{n}\sum_{k\neq k^{*}}\mathbf{E}T_{k}(n)\mathbf{E}\left|x_{k,\,t}-\mu_{k^{*}}\right|+\gamma pM^{p-1}\mathbf{E}\left|\frac{1}{T_{k^{*}}\left(n\right)}\sum_{t=1}^{T_{k^{*}}\left(n\right)}\left(x_{k^{*},\,t}-\mu_{k^{*}}\right)\right|.
\end{align*}
Furthermore, the first term above is bounded by $\frac{\gamma pM^{p}}{n}\sum_{k\neq k^{*}}\mathbf{E}T_{k}(n)$.
Restricted to the subset $\Omega_{n}^{M}$, the second term above satisfies the bound
\begin{align*}
\left|\frac{1}{T_{k^{*}}\left(n\right)}\sum_{t=1}^{T_{k^{*}}\left(n\right)}\left(x_{k^{*},\,t}-\mu_{k^{*}}\right)\right|=\left|\overline{x}_{k^{*},\,T_{k^{*}}\left(n\right)}-\mu_{k^{*}}\right| & \leq M\sqrt{\frac{\log\frac{4n^{2}K}{\delta}}{2T_{k^{*}}\left(n\right)}}.
\end{align*}
Meanwhile, $|\overline{x}_{k^{*},\,T_{k^{*}}(n)}-\mu_{k^{*}}|$ is bounded by $M$, and so
\begin{equation*}
\mathbf{E}\left|\frac{1}{T_{k^{*}}\left(n\right)}\sum_{t=1}^{T_{k^{*}}\left(n\right)}\left(x_{k^{*},\,t}-\mu_{k^{*}}\right)\right|\leq\left(1-\frac{\delta}{n}\right)M\sqrt{\frac{\log\frac{4n^{2}K}{\delta}}{2T_{k^{*}}\left(n\right)}}+\frac{\delta}{n}M.
\end{equation*}
From this reasoning, we see that both the last two terms in (\ref{MD theorem 1}) are bounded which gives the desired conclusion.
\hfill\rule{1.5mm}{3mm}
\end{proof}

\subsection{Shortfall risk measure}\label{Shortfall risk measure proof}

In Lemmas \ref{lemma 10}-\ref{lemma 12}, we first establish that
$\widehat{\rho}_{k,\,n}^{S}$ is an M estimator [\citep{dasgupta2008asymptotic} Definition 17.1] of $\rho_{k}^{S}$. Then we derive an upper bound for $|\widehat{\rho}_{k,\,n}^{S}-\rho_{k}^{S}|$, based on which, we establish the upper bound on the number of pulls of each sub-optimal arm. Then we prove Theorem \ref{shortfall theorem} in the end of the section.

\noindent\textbf{Step 1: Bounding} $|\widehat{\rho}_{k,\,n}^{S}-\rho_{k}^{S}|$ \textbf{by the concentration results for M estimators}. Let $G_{k}(\kappa)\triangleq\mathbf{E}[l(X_{k}-\kappa)]$ for notational convenience, based on Assumption \ref{shortfall assumption 1}, we conclude that $G_{k}(\kappa)$ obeys the following property.
\begin{lemmastar}\label{lemma 10}
The derivative of $G_{k}(\kappa)$, denoted by $G_{k}^{'}(\kappa)$, exists in $(0,\:M)$
and $[G_{k}^{'}(SF_{k})]^{-1}$ is upper bounded by a constant $M_{G}$, for all $k=1,\ldots,K$.
\end{lemmastar}
\noindent The following lemma establishes that $\widehat{\rho}_{k,\,n}^{S}$ is an M
estimator and converges to $\rho_{k}^{S}$ a.s.
\begin{lemmastar}\label{lemma 11}
Suppose Assumption \ref{shortfall assumption 1} holds, $\widehat{\rho}_{k,\,n}^{S}$ is the unique solution in $\kappa$ of
\begin{equation*}
\int_{0}^{M}l(X-\kappa)dF_{i,\,n}(X)=0.
\end{equation*}
Then $\widehat{\rho}_{k,\,n}^{S}$ is an M estimator of $\rho_{k}^{S}$. Further,
$\rho_{k}^{S}$ is the unique solution of $G_{k}(\kappa)=0$,
and $\widehat{\rho}_{k,\,n}^{S} \overset{a.s.}{\rightarrow} \rho_{k}^{S}$.
\end{lemmastar}
Based on the conclusion in Lemma \ref{lemma 11}, we may obtain the Bahadur representation for $|\widehat{\rho}_{k,\,n}^{S}-\rho_{k}^{S}|$. Subsequently, we may establish the confidence bound on $\widehat{\rho}_{k,\,n}^{S}$ which is given in the following lemma.
\begin{lemma}\label{lemma 12}
For all $n\geq 1$, the following bound on $|\widehat{\rho}_{k,\,n}^{S}-\rho_{k}^{S}|$ holds with probability at least $1-\delta$
\begin{equation}\label{lemma 12-1}
|\widehat{\rho}_{k,\,n}^{S}-\rho_{k}^{S}| \leq
2M_{l}M_{G}\sqrt{\frac{\log\frac{4}{\delta}}{2n}}.
\end{equation}
Furthermore,
\begin{equation}\label{lemma 12-2}
\mathbf{E}\left(|\widehat{\rho}_{k,\,n}^{S}-\rho_{k}^{S}|\right)\leq2M_{l}M_{G}\sqrt{\frac{\log\frac{4}{\delta}}{2n}}+\delta M.
\end{equation}
\end{lemma}

\begin{proof}
From the conclusion of Lemma \ref{lemma 11}, $\widehat{\rho}_{k,\,n}^{S}$ is an M estimator that converges
to $\rho_{k}^{S}$. The Bahadur representation for
$\rho_{k}^{S}-\widehat{\rho}_{k,\,n}^{S}$ is given as below [\citep{dasgupta2008asymptotic} Theorem 17.3]
\begin{equation*}
\rho_{k}^{S}-\widehat{\rho}_{k,\,n}^{S}
=\frac{1}{n}\sum_{t=1}^{n}\frac{l\left(x_{k,\,t}-\rho_{k}^{S}\right)}{G_{k}^{'}\left(\rho_{k}^{S}\right)}+r_{n},
\end{equation*}
where $r_{n}=o_{p}(1/\sqrt{n})$. Noting that $\mathbf{E}[l(x_{k,\,t}-\rho_{k}^{S})]=0$
and each $l(x_{k,\,t}-\rho_{k}^{S})[G_{k}^{'}(\rho_{k}^{S})]^{-1}$
is bounded by $M_{l}M_{G}$, by Hoeffding's inequality we have
\begin{equation*}
P\left\{ \left|\frac{1}{n}\sum_{t=1}^{n}\frac{l\left(x_{k,\,t}-\rho_{k}^{S}\right)}{G_{k}^{'}\left(\rho_{k}^{S}\right)}\right|\geq M_{l}M_{G}\sqrt{\frac{\log\frac{4}{\delta}}{2n}}\right\} \leq\frac{\delta}{2}.
\end{equation*}
Then, we have the following bound on the probability that
$|\widehat{\rho}_{k,\,n}^{S}-\rho_{k}^{S}|$ exceeds a certain threshold
\begin{equation*}
P\left\{ \left|\widehat{\rho}_{k,\,n}^{S}-\rho_{k}^{S}\right|\geq2M_{l}M_{G}\sqrt{\frac{\log\frac{4}{\delta}}{2n}}\right\} \leq\delta.
\end{equation*}
Conclusion (\ref{lemma 12-1}) follows from this bound. As
$\widehat{\rho}_{k,\,n}^{S}\in\left[0,\:M\right]$
and $\rho_{k}^{S} \in \left[0,\:M\right]$, we know that
$\left|\widehat{\rho}_{k,\,n}^{S}-\rho_{k}^{S}\right|\leq M$, and so
conclusion (\ref{lemma 12-2}) follows.
\hfill\rule{1.5mm}{3mm}
\end{proof}
Lemma \ref{lemma 12} establishes an probability upper bound on
$|\widehat{\rho}_{k,\,n}^{S}-\rho_{k}^{S}|$,
based on which, we can establish the index policy for the shortfall case.

\noindent\textbf{Step 2: Bounding the number of pulls of each sub-optimal arm}. The following lemma gives bounds on $T_{k}(n)$ and $\mathbf{E}T_{i}(n)$ in probability. We define the following key subset
\begin{equation*}
\Omega_{n}^{S}\triangleq\left\{ \forall\:k=1,\ldots,K,\;\forall t=1,\ldots,n,\:s.t.\:|\widehat{\rho}_{k,\,n}^{S}-\rho_{k}^{S}|\leq2M_{l}M_{G}\sqrt{\frac{\log\frac{4n^{2}K}{\delta}}{2t}}\right\}.
\end{equation*}
As before, subset $\Omega_{n}^{S}$ collects all the realizations where the empirical risk of each arm is close enough to its real risk at each of the first $n$ times. The following lemma gives a lower bound on the probability of $\Omega_{n}^{S}$.
\begin{lemmastar}\label{lemma 13}
For all $n\geq 1$, we have $\mathbf{P}\{\Omega_{n}^{S}\} \geq1-{\delta}/{n}$.
\end{lemmastar}
\noindent Through the proofs of Lemma \ref{lemma 12} and Lemma \ref{lemma 13}, we have shown that $\varepsilon_{\rho}^{S}$ satisfies Condition 1. In the following lemma, we establish an upper bound on the number of pulls of sub-optimal arms.
\begin{lemma}\label{lemma 14}
For any sub-optimal arm $k\neq k^{*}$,
the following relation holds with probability at least $1-{\delta}/{n}$
\begin{equation}\label{lemma 14-1}
T_{k}\left(n\right)\leq\frac{8M_{l}M_{G}}{\Delta_{k}^{S}}\log\frac{4n^{2}K}{\delta}.
\end{equation}
Furthermore,
\begin{equation}\label{lemma 14-2}
\mathbf{E}T_{k}\left(n\right)\leq\left(1-\frac{\delta}{n}\right)\frac{8M_{l}M_{G}}{\Delta_{k}^{S}}\log\frac{4n^{2}K}{\delta}+\delta.
\end{equation}
Specifically, $\mathbf{E}T_{k}(n)\leq O(\log n)$.
\end{lemma}
\begin{proof}
The main idea of the proof is the same as Lemma \ref{lemma 4}. As
$|\widehat{\rho}_{k,\,n}^{S}-\rho_{k}^{S}|$ has been bounded in Lemma \ref{lemma 12}, we can directly establish an upper bound on the number of pulls of each sub-optimal arm in probability using the relation of indices between each each sub-optimal arm and the optimal arm. Restricted to the set $\Omega_{n}^{S}$, the following inequality holds for each arm, including the optimal arm $k^{*}$
\begin{equation*}
-2M_{l}M_{G}\sqrt{\frac{\log\frac{4n^{2}K}{\delta}}{2T_{k}(n)}}
\leq \widehat{\rho}_{k,\,n}^{S}-\rho_{k}^{S}
\leq 2M_{l}M_{G}\sqrt{\frac{\log\frac{4n^{2}K}{\delta}}{2T_{k}(n)}}.
\end{equation*}
At time $n+1$, if any sub-optimal arm $k\neq k^{*}$ is pulled, then we must have
\begin{align*}
\rho_{k}^{S}-4M_{l}M_{G}\sqrt{\frac{\log\frac{4n^{2}K}{\delta}}{2T_{k}\left(n\right)}}\leq B_{k,\,n}^{S} & \leq B_{k^{*},\,n}^{S}\leq \rho_{k^{*}}^{S}.
\end{align*}
Conclusion (\ref{lemma 14-1}) then follows directly from the above inequality
and conclusion (\ref{lemma 14-2}) follows by noting that $T_{k}(n)$ is automatically bounded by $n$.
\hfill\rule{1.5mm}{3mm}
\end{proof}

\noindent\textbf{Step 3: Establishing an upper bound on the pseudo regret}. Based on  previous results we have upper bounds both on
$|\widehat{\rho}_{k,\,n}^{S}-\rho_{k}^{S}|$ and the number of pulls of sub-optimal arms. We may now give the full derivation of the upper bound on the pseudo regret in Theorem 3.
\begin{proof}[\textup{\textbf{Proof of Theorem \ref{shortfall theorem}}}]
First, we decompose and bound $|\widehat{\rho}_{\varphi,\,n}^{S}-\rho_{k^{*}}^{S}|$ with
\begin{equation}\label{Shortfall theorem 2}
\left|\widehat{\rho}_{\varphi,\,n}^{S}-\rho_{k^{*}}^{S}\right|
\leq \left|\widehat{\rho}_{k^{*},\,n}^{S}-\widehat{\rho}_{\varphi,\,n}^{S}\right|
+ \left|\widehat{\rho}_{k^{*},\,n}^{S}-\rho_{k^{*}}^{S}\right|.
\end{equation}
In the rest of the proof, we will bound both terms in the right side of (\ref{Shortfall theorem 2}).  Let $G_{k,\,n}(\kappa) \triangleq \frac{1}{n}\sum_{t=1}^{n}l(x_{k,\,t}-\kappa)$ and $G_{\varphi,\,n}(\kappa) \triangleq \frac{1}{n}\sum_{t=1}^{n}l(x_{I_{t},\,t}-\kappa)$.
We first bound the term
$|\widehat{\rho}_{k^{*},\,n}^{S}-\widehat{\rho}_{\varphi,\,n}^{S}|$ in (\ref{Shortfall theorem 2}). Based on Assumption \ref{shortfall assumption 1}, $G_{\varphi,\,n}(\kappa)$
is monotone decreasing in $\kappa$ and $G_{\varphi,\,n}^{'}\left(\kappa\right)$
exists in $[0,\:M]$ and is upper bounded by $-m_{l}$. By the differential mean value theorem, there exists an $\xi_{n}$
between $\widehat{\rho}_{k^{*},\,n}^{S}$ and $\widehat{\rho}_{\varphi,\,n}^{S}$ that satisfies
\begin{equation*}
\left|G_{\varphi,\,n}\left(\widehat{\rho}_{k^{*},\,n}^{S}\right)-G_{\varphi,\,n}\left(\widehat{\rho}_{\varphi,\,n}^{S}\right)\right|=\left|G_{\varphi,\,n}^{'}\left(\xi_{n}\right)\left(\widehat{\rho}_{k^{*},\,n}^{S}-\widehat{\rho}_{\varphi,\,n}^{S}\right)\right|,
\end{equation*}
so we have
\begin{equation*}
\left|\widehat{\rho}_{k^{*},\,n}^{S}-\widehat{\rho}_{\varphi,\,n}^{S}\right| \leq \frac{1}{m_{l}}\left|G_{\varphi,\,n}\left(\widehat{\rho}_{k^{*},\,n}^{S}\right)-G_{\varphi,\,n}\left(\widehat{\rho}_{\varphi,\,n}^{S}\right)\right|.
\end{equation*}
From Lemma \ref{lemma 11}, we know that $\widehat{\rho}_{\varphi,\,n}^{S}$
is the unique solution of $G_{\varphi,\,n}(\kappa)$, which
means that $G_{\varphi,\,n}(\widehat{\rho}_{\varphi,\,n}^{S})=0$.
Restricted to the subset $\Omega_{n}^{S}$, we have
\begin{align*}
\left|G_{\varphi,\,n}\left(\widehat{\rho}_{k^{*},\,n}^{S}\right)\right|
&  =\left|\frac{1}{n}\sum_{t=1}^{n}l\left(x_{I_{t},\,t}-\widehat{\rho}_{k^{*},\,n}^{S}\right)\right| \\
& =\left|\frac{1}{n}\sum_{t=1}^{T_{k^{*}}\left(n\right)}l\left(x_{k^{*},\,t}-\widehat{\rho}_{k^{*},\,n}^{S}\right)+\frac{1}{n}\sum_{k\neq k^{*}}\sum_{t=1}^{T_{k}\left(n\right)}l\left(x_{k,\,t}-\widehat{\rho}_{k^{*},\,n}^{S}\right)\right| \\
& =\left|\frac{T_{k^{*}}\left(n\right)}{n}G_{k^{*},\,T_{k^{*}}\left(n\right)}\left(\widehat{\rho}_{k^{*},\,n}^{S}\right)+\frac{1}{n}\sum_{k\neq k^{*}}\sum_{t=1}^{T_{k}\left(n\right)}l\left(x_{k,\,t}-\widehat{\rho}_{k^{*},\,n}^{S}\right)\right| \\
& \leq\frac{\sum_{k\neq k^{*}}T_{k}\left(n\right)}{n}M_{l} \\
& \leq\sum_{k\neq k^{*}}\frac{8M_{l}^{2}M_{G}}{n\Delta_{k}^{S}}\log\frac{4n^{2}K}{\delta},
\end{align*}
where the first inequality follows by noting that $l$ is bounded
by $M_{l}$. The above inequality implies
\begin{equation*}
\left|G_{\varphi,\,n}\left(\widehat{\rho}_{k^{*},\,n}^{S}\right)-G_{\varphi,\,n}\left(\widehat{\rho}_{\varphi,\,n}^{S}\right)\right|\leq\sum_{k\neq k^{*}}\frac{8M_{l}^{2}M_{G}}{n\Delta_{k}^{S}}\log\frac{4n^{2}K}{\delta},
\end{equation*}
and so $|\widehat{\rho}_{k^{*},\,n}^{S}-\widehat{\rho}_{\varphi,\,n}^{S}|$ is bounded by
\begin{equation*}
\left|\widehat{\rho}_{k^{*},\,n}^{S}-\widehat{\rho}_{\varphi,\,n}^{S}\right|
\leq \sum_{k\neq k^{*}}\frac{8M_{l}^{2}M_{G}}{nm_{l}\Delta_{k}^{S}}\log\frac{4n^{2}K}{\delta}.
\end{equation*}
Meanwhile, restricted to the event $\Omega_{n}^{S}$, the second term in (\ref{Shortfall theorem 2}) is bounded by
\begin{equation*}
\left|\widehat{\rho}_{k^{*},\,n}^{S}-\rho_{k^{*}}^{S}\right|\leq2M_{l}M_{G}\sqrt{\frac{\log\frac{4n^{2}K}{\delta}}{2T_{k}\left(n\right)}}.
\end{equation*}
Finally, we have
\begin{align*}
\left|\widehat{\rho}_{\varphi,\,n}^{S}-\rho_{k^{*}}^{S}\right|
& \leq \left|\widehat{\rho}_{k^{*},\,n}^{S}-\widehat{\rho}_{\varphi,\,n}^{S}\right|
+ \left|\widehat{\rho}_{k^{*},\,n}^{S}-\rho_{k^{*}}^{S}\right| \\
 & \leq\sum_{k\neq k^{*}}\frac{8M_{l}^{2}M_{G}}{nm_{l}\Delta_{k}^{S}}\log\frac{4n^{2}K}{\delta}+2M_{l}M_{G}\sqrt{\frac{\log\frac{4n^{2}K}{\delta}}{2T_{k}\left(n\right)}}.
\end{align*}
As for the proofs of Theorem \ref{CVaR theorem} and \ref{MD theorem}, the conclusion
follows by noting that both $\widehat{\rho}_{\varphi,\,n}^{S}$ and $\rho_{k^{*}}^{S}$
lie in $\left[0,\:M\right]$.
\hfill\rule{1.5mm}{3mm}
\end{proof}

\section{Concluding remarks}\label{Concluding remarks}

In this work, we focus on risk-aware MAB where the objective is a coherent risk measure. We introduce three
specific risk measures which are widely investigated in the literature,
yet not studied in MAB problems. As our main contribution, we construct an index-based policy for risk-averse MAB and bound its
pseudo regret for our three specific risk measures. In particular, we show that the upper bound on the pseudo regret is of the order of $O(\sqrt{{\log n}/{n}})$ which is different from the order $O({\log n}/{n})$ in risk-neutral case. In our discussion in Section \ref{Discussion}, we show that this is because the relation between the pseudo regret and the number of pulls of each sub-optimal arm is nonlinear. Meanwhile, we notice that when the risk measures become expectations (by changing the corresponding coefficients), both orders in the risk-averse case and the risk neutral case are consistent. We note that our index policy has a simple form and is thus practical and versatile.

In future work, following the same procedure in this paper, we may extend the scope of our study to
incorporate more risk measures. Our scheme in this paper is actually quite general, and only depends on being able to obtain confidence
bounds for empirical estimation of risk measures. Moreover, in this work, we need to make specific assumptions on
the distribution of each arm as well as on risk measures themselves. For example, for CVaR, we require a continuous CDF. The possibility of relaxing these assumptions is worthy of further study.

\begin{acknowledgement}
This research was supported by MOE Tier I grant WBS R266-000-087-112
and MOE Tier I grant WBS R266-000-104-112.
\end{acknowledgement}

\section*{Appendix}\label{Appendix}

\begin{proof}[\textup{\textbf{Proof of Lemma \ref{lemma 3}}}]
By Hoeffding's inequality and noting that $M$ is
the uniform upper bound for all the arms, $\forall t=1,\ldots,n$,
we have
\begin{equation*}
\mathbf{P}\left\{ \left|\frac{1}{t}\sum_{s=1}^{t}(x_{k,\,s}-\eta_{k,\,t})_{+}-\mathbf{E}\left[(X_{k}-\eta_{k,\,t})_{+}\right]\right|\geq M\sqrt{\frac{\log\frac{2n^{2}K}{\delta}}{2t}}\right\} \leq\frac{\delta}{n^{2}K}.
\end{equation*}
Thus, by replacing the term $\left[f_{k}(\nu_{k\alpha})\right]^{-1}$
by $m\left(\alpha\right)$ in Lemma \ref{lemma 2}, we have
\begin{align*}
\mathbf{P}\left\{ \left(\Omega_{n}^{C}\right)^{c}\right\} \leq & \sum_{k=1}^{K}\sum_{t=1}^{n}\mathbf{P}\left\{ |\eta_{k,\,t}-\eta_{k}|\geq2m\left(\alpha\right)\sqrt{\frac{\log\frac{2n^{2}K}{\delta}}{2t}}\right\} \\
 & +\sum_{k=1}^{K}\sum_{t=1}^{n}\mathbf{P}\left\{ \left|\frac{1}{t}\sum_{s=1}^{t}(x_{k,\,s}-\eta_{k,\,t})_{+}-\mathbf{E}\left[(X_{k}-\eta_{k,\,t})_{+}\right]\right|\geq M\sqrt{\frac{\log\frac{2n^{2}K}{\delta}}{2t}}\right\} \\
 \leq & \frac{2\delta}{n}+\frac{\delta}{n}=\frac{3\delta}{n}.
\end{align*}
The conclusion follows.
\hfill\rule{1.5mm}{3mm}
\end{proof}

\begin{proof}[\textup{\textbf{Proof of Lemma \ref{lemma 8}}}]
When $x\geq1$, we have $c\leq ax+bx^{\frac{1}{p}}\leq\left(a+b\right)x$, so
\begin{equation*}
x\geq\max\left\{ 1,\:\frac{c}{a+b}\right\}.
\end{equation*}
When $x\leq1$, we have $c\leq ax+bx^{\frac{1}{p}}\leq\left(a+b\right)x^{\frac{1}{p}}$, so
\begin{equation*}
\min\left\{ 1,\:\left(\frac{c}{a+b}\right)^{p}\right\} \leq x\leq1.
\end{equation*}
And thus we conclude that $x\geq\min\{ 1,\:\frac{c}{a+b},\:(\frac{c}{a+b})^{p}\}$.
\hfill\rule{1.5mm}{3mm}
\end{proof}

\begin{proof}[\textup{\textbf{Proof of Lemma \ref{lemma 7}}}]
By Hoeffding's inequality, and so we have
\begin{equation*}
\mathbf{P}\left\{ |\overline{x}_{k,\,t}-\mu_{k}|\geq M\sqrt{\frac{\log\frac{4n^{2}K}{\delta}}{2t}}\right\} \leq\frac{\delta}{2n^{2}K}.
\end{equation*}
Note that $\left|x_{k,\,s}-\mu_{k}\right|^{p}$ is bounded by $M^{p}$, we have
\begin{equation*}
\mathbf{P}\left\{ \left|\frac{1}{t}\sum_{s=1}^{t}\left|x_{k,\,s}-\mu_{k}\right|^{p}-\mathbf{E}\left|X_{k}-\mu_{k}\right|^{p}\right|\geq M^{p}\sqrt{\frac{\log\frac{4n^{2}K}{\delta}}{2t}}\right\} \leq\frac{\delta}{2n^{2}K}.
\end{equation*}
The probability of the event $(\Omega_{n}^{M})^{c}$ is then bounded above by
\begin{align*}
\mathbf{P}\left\{ \left(\Omega_{n}^{M}\right)^{c}\right\} \leq & \sum_{k=1}^{K}\sum_{t=1}^{n}P\left\{ |\overline{x}_{k,\,t}-\mu_{k}|\geq M\sqrt{\frac{\log\frac{4n^{2}K}{\delta}}{2t}}\right\} \\
 & +\sum_{k=1}^{K}\sum_{t=1}^{n}P\left\{ \left|\frac{1}{t}\sum_{s=1}^{t}\left|x_{k,\,s}-\mu_{k}\right|^{p}-\mathbf{E}\left|X_{k}-\mu_{k}\right|^{p}\right|\geq M^{p}\sqrt{\frac{\log\frac{4n^{2}K}{\delta}}{2t}}\right\} \\
 \leq & \frac{\delta}{2n}+\frac{\delta}{2n}=\frac{\delta}{n}.
\end{align*}
The conclusion then follows.
\hfill\rule{1.5mm}{3mm}
\end{proof}

\begin{proof}[\textup{\textbf{Proof of Lemma \ref{lemma 10}}}]
We can directly compute $G_{k}(\kappa)$ as
\begin{align*}
G_{k}^{'}\left(\kappa\right) & =\lim_{\Delta\kappa\rightarrow0}\frac{G_{k}\left(\kappa+\Delta\kappa\right)-G_{k}\left(\kappa\right)}{\Delta\kappa}\\
 & =\lim_{\Delta\kappa\rightarrow0}\frac{\int_{0}^{M}\left[l\left(X_{k}-\kappa-\Delta\kappa\right)-l\left(X_{k}-\kappa\right)\right]dF_{k}\left(X_{k}\right)}{\Delta\kappa}\\
 & =\lim_{\Delta\kappa\rightarrow0}\frac{\int_{0}^{M}\left[l^{'}\left(X_{k}-\kappa\right)\Delta\kappa+o\left(\Delta\kappa\right)\right]dF_{k}\left(X_{k}\right)}{\Delta\kappa}\\
 & =\int_{0}^{M}l^{'}\left(X_{k}-\kappa\right)dF_{k}\left(X_{k}\right).
\end{align*}
Based on Assumption \ref{shortfall assumption 1}, $l^{'}(X_{k}-\kappa)$ is upper bounded by $C_{l}$ and is also lower bounded by noting
that $l(t)$ is strictly increasing in a closed interval $[-M,\:M]$. So, $G_{k}^{'}(\kappa)$ exists
in $(0,\:M)$ and is both upper bounded and lower bounded. The conclusion follows.
\hfill\rule{1.5mm}{3mm}
\end{proof}

\begin{proof}[\textup{\textbf{Proof of Lemma \ref{lemma 11}}}]
Based on Assumption \ref{shortfall assumption 1}, $l(x_{k,\,t}-\kappa)$ is continuous
and strictly monotone decreasing in $\kappa$ for any observation $x_{k,\,t}$.
As $X_{k}$ is bounded in $[0,\:M]$,
noting the monotonicity of $l(t)$ and $l(0)=0$, we have $l\left(x_{k,\,t}\right)\geq0$ and $l\left(x_{k,\,t}-M\right)\leq0$,
so $\widehat{\rho}_{k,\,n}^{S}$ is the unique solution of
\begin{equation*}
\frac{1}{n}\sum_{t=1}^{n}l\left(x_{i,\,t}-\kappa\right)=0,
\end{equation*}
which means $\widehat{\rho}_{k,\,n}^{S}$ is an M estimator. Similarly,
$G_{k}(\kappa)$ is continuous and strictly monotone decreasing because
\begin{equation*}
G_{i}\left(\kappa_{1}\right)-G_{i}\left(\kappa_{2}\right)=\mathbf{E}\left[l\left(X_{i}-\kappa_{1}\right)-l\left(X_{i}-\kappa_{2}\right)\right]>0
\end{equation*}
for any $0\leq\kappa_{1}<\kappa_{2}\leq M$. Let the Lipschitz constant of $l$ be $C_{l}$, then
\begin{align*}
\left|G_{k}\left(\kappa_{1}\right)-G_{k}\left(\kappa_{2}\right)\right| & =\left|\mathbf{E}\left[l\left(X_{k}-\kappa_{1}\right)\right]-\mathbf{E}\left[l\left(X_{k}-\kappa_{2}\right)\right]\right|\\
 & \leq\mathbf{E}\left|l\left(X_{k}-\kappa_{1}\right)-l\left(X_{k}-\kappa_{2}\right)\right|\\
 & \leq\mathbf{E}\left|C_{l}\left(\kappa_{1}-\kappa_{2}\right)\right|=C_{l}\left|\kappa_{1}-\kappa_{2}\right|,
\end{align*}
so $\rho_{k}^{S}$ is the solution of $G_{k}\left(\kappa\right)=0$.
From [\citep{dasgupta2008asymptotic} Theorem 17.1], we can conclude
that $\widehat{\rho}_{k,\,n}^{S} \overset{a.s.}{\rightarrow} \rho_{k}^{S}$.
\hfill\rule{1.5mm}{3mm}
\end{proof}

\begin{proof}[\textup{\textbf{Proof of Lemma \ref{lemma 13}}}]
Based on Lemma \ref{lemma 12}, we have
\begin{equation*}
\mathbf{P}\left\{ \left|\widehat{\rho}_{k,\,t}^{S}-\rho_{k}^{S}\right|
\geq 2M_{l}M_{G}\sqrt{\frac{\log\frac{4n^{2}K}{\delta}}{2t}}\right\}
\leq \frac{\delta}{n^{2}K}.
\end{equation*}
\noindent It follows that
\begin{align*}
\mathbf{P}\{(\Omega_{n}^{S})^{c}\}
& \leq\sum_{k=1}^{K}\sum_{t=1}^{n}P\left\{ \left|\widehat{\rho}_{k,\,t}^{S}-\rho_{k}^{S}\right|\geq2M_{l}M_{G}\sqrt{\frac{\log\frac{4n^{2}K}{\delta}}{2t}}\right\}\\
& \leq nK\cdot\frac{\delta}{n^{2}K}=\frac{\delta}{n},
\end{align*}
from which we conclude the proof.
\hfill\rule{1.5mm}{3mm}
\end{proof}

\bibliography{ref}
\end{document}